\documentclass[11pt,reqno]{amsart}
\usepackage[letterpaper,textwidth=6.05in,textheight=8.8in,centering]{geometry}
\usepackage{amsmath,amssymb,amsthm}
\usepackage{microtype}
\usepackage{needspace}
\usepackage{booktabs}
\usepackage[hidelinks]{hyperref}

\newtheorem{theorem}{Theorem}[section]
\newtheorem{proposition}[theorem]{Proposition}
\newtheorem{lemma}[theorem]{Lemma}
\newtheorem{corollary}[theorem]{Corollary}
\theoremstyle{definition}

\theoremstyle{remark}

\newcommand{\V}{\mathsf{V}}
\DeclareMathOperator{\Id}{Id}
\DeclareMathOperator{\occ}{occ}

\newcommand{\Pf}{P_f}
\newcommand{\HH}{\mathbb{H}}
\newcommand{\FF}{\mathbb{F}}
\newcommand{\KK}{\mathbb{K}}
\newcommand{\GG}{\mathbb{G}}
\newcommand{\bu}{\mathbf{u}}
\newcommand{\bv}{\mathbf{v}}
\newcommand{\bw}{\mathbf{w}}
\newcommand{\bp}{\mathbf{p}}
\newcommand{\bq}{\mathbf{q}}
\newcommand{\bd}{\mathbf{d}}
\newcommand{\bt}{\mathbf{t}}
\newcommand{\br}{\mathbf{r}}
\newcommand{\Spow}[1]{S_c^*(a^{#1})}
\newcommand{\Slin}[1]{S_c^*(a_1\cdots a_{#1})}
\newcommand{\Spowinf}{S_c^*(\{a\}^{+})}
\newcommand{\Slininf}{S_c^*(\{a_1\cdots a_m:m\geq1\})}

\hypersetup{
 pdftitle={Finite bases and joins for semirings defined by the divisibility order},
 pdfauthor={Xiaolei Shao and Zidong Gao},
 pdfsubject={Identities and varieties of additively idempotent semirings},
 pdfkeywords={semiring, divisibility order, finite basis, join, hypergraph}
}
\title[Finite bases and joins of semirings]
 {Finite bases and joins for semirings defined by the divisibility order}
\author{Xiaolei Shao}
\author{Zidong Gao}
\date{}

\begin{document}
\begin{abstract}
We study additively idempotent semirings obtained from commutative words
by equipping their subwords with the divisibility order. Every finite
semiring associated with a power of one letter is finitely based, whereas
one associated with a linear word is finitely based exactly when the word
has length at most two. A hypergraph preservation lemma yields the
nonfinite basis result and extends it to intervals of varieties. We
establish a sharp containment criterion between the power and linear
families, and determine the finite basis property of every join of two
varieties generated by one member of each family. The unrestricted power
family generates the nonfinitely based max-plus variety. In contrast,
the unrestricted linear family has a finite basis, as does its join with
each finite power member. We also realize a previously known six-element
limit semiring as a quotient of a subsemiring of the eight-element
linear-word semiring. This gives a proper nonfinitely based subvariety
and resolves the corresponding minimality question.
\end{abstract}
\maketitle

\section{Introduction}

An additively idempotent semiring (an ai-semiring) is an algebra $(S,+,\cdot)$ whose
additive reduct is a commutative idempotent semigroup, whose
multiplicative reduct is a semigroup, and whose operations satisfy both
distributive laws. We work in the signature $(+,\cdot)$. For an algebra
$S$, let $\V(S)$ denote the variety generated by $S$. A variety, or an
algebra generating it, is \emph{finitely based} if its identities follow
from a finite set of identities; otherwise it is \emph{nonfinitely
based}.

Gao introduced the semirings $S_c^*(W)$ by using the divisibility order
on commutative word semigroups \cite[Section~4]{Gao}. Their nonzero
elements are the nonempty commutative subwords of words in $W$.
Multiplication is truncated word multiplication, and addition is the
supremum in the divisibility order, with an absorbing zero adjoined.
This construction connects word semirings with the max-plus semiring
and with the finitely based three-element semiring $S_{53}$.
In particular, $\Slin{k}\in\V(S_{53})$ for every $k\geq1$, and,
for $k\geq3$, the interval from $\V(\Slin{k})$ to the
$(k+1)$-nilpotent subvariety of $\V(S_{53})$ has the cardinality
of the continuum \cite[Lemma~4.4 and Theorem~4.5]{Gao}.

We determine the finite basis property for two basic families in this
construction. For every finite $n\geq1$, the semiring $\Spow{n}$
is finitely based. For every finite $k\geq1$, the semiring
$\Slin{k}$ is finitely based if and only if $k\leq2$. The proof
for powers uses a bounded selection of words from the upper term of
an inequality. The proof for linear words of length at least three uses
high-girth hypergraphs and a preservation property for equational
deductions.

The two families are related by the exact containment criterion
\begin{equation}\label{eq:introcontainment}
 \V(\Slin{k})\subseteq\V(\Spow{n})
 \quad\Longleftrightarrow\quad n\geq2k-1.
\end{equation}
An explicit homomorphic image of a subsemiring of a direct power
establishes the positive direction. A square-versus-product inequality
gives the converse. These same inequalities, combined with the
hypergraph preservation lemma, determine the finite basis property of
mixed joins:
\begin{equation}\label{eq:introjoin}
 \V(\Spow{n})\vee\V(\Slin{k})
 \text{ is finitely based}
 \quad\Longleftrightarrow\quad
 k\leq2\ \text{or}\ n\geq2k-1.
\end{equation}
We also obtain intervals consisting entirely of nonfinitely based
varieties, and strict intermediate nonfinitely based varieties between
successive finitely based members of the power chain.

For the infinite families, the relevant algebras are
$\Spowinf$ and $\Slininf$. We prove
\begin{align}
 \V(\Spowinf)
 &=\bigvee_{n\geq1}\V(\Spow{n})
   =\V((\mathbb N_0,\max,+)),\label{eq:intropowerjoin}\\
 \V(\Slininf)
 &=\bigvee_{k\geq1}\V(\Slin{k}).\label{eq:introlinearjoin}
\end{align}
The first variety is nonfinitely based by the max-plus theorem of
Aceto, \'Esik and Ing\'olfsd\'ottir \cite{AcetoEsikIngolf}.
The second has an explicit finite basis. For every finite $n$, its join
with $\V(\Spow{n})$ is also finitely based. Thus the finite
basis property changes in opposite directions in the two unbounded
families.

Shao, Ren and Gao proved that $TR_6$ generates a limit variety
\cite[Theorem~4.6]{ShaoRenGao}. We construct a quotient of a
seven-element subsemiring of $S_c^*(abc)$ which is isomorphic to
$TR_6$. Consequently,
\[
 \V(TR_6)\subsetneq\V(S_c^*(abc)).
\]
This gives a negative answer to the minimality question in
\cite[Problem~6.1]{Gao}.

Section~\ref{sec:preliminaries} fixes the notation for terms and
inequalities. Sections~\ref{sec:semiring}--\ref{sec:unrestricted}
describe the algebras and prove the finite basis results for finite
powers and the unrestricted linear family. The hypergraph arguments
occupy Sections~\ref{sec:hypergraphmethods} and
\ref{sec:mainproof}. Section~\ref{sec:finitejoins} proves
\eqref{eq:introcontainment} and \eqref{eq:introjoin}, including
the small boundary cases. Section~\ref{sec:infinitejoins} treats the
infinite joins. Section~\ref{sec:intervals} gives the interval
consequences and the quotient construction for $TR_6$.

\section{Terms, inequalities and equational deductions}
\label{sec:preliminaries}

We use the notation for words and terms from \cite[Section~2]{Gao}.
Let $X$ be a countably infinite set of variables. Write $X^+$ for the
free semigroup on $X$, $X_c^+$ for the free commutative semigroup on
$X$, and $X_c^*=X_c^+\cup\{1\}$ for the free commutative monoid.
Here $1$ is the empty word. Throughout the proof, words belong to
$X_c^+$; the symbol $1$ indicates an absent multiplicative context.
For a positive integer $n$, write $[n]=\{1,\ldots,n\}$.

For a word $\bp$, the notation $\occ(x,\bp)$ denotes the number of
occurrences of $x$ in $\bp$, and $c(\bp)$ denotes its content. Thus,
if
\[
 \bp=x_1^{m_1}\cdots x_s^{m_s},
 \qquad m_i\geq 1,
\]
where $x_1,\ldots,x_s$ are distinct, then
\[
 c(\bp)=\{x_1,\ldots,x_s\},
 \qquad
 \ell(\bp)=m_1+\cdots+m_s.
\]
Here $\ell(\bp)$ denotes the length of $\bp$. A word is
\emph{linear} if $\occ(x,\bp)=1$ for every $x\in c(\bp)$.
For the formal empty word we put $c(1)=\varnothing$ and $\ell(1)=0$.

By distributivity, every ai-semiring term can be written as a finite
nonempty sum of words. In the commutative case we write
\begin{equation}\label{eq:term}
 \bu=\bu_1+\cdots+\bu_m,
 \qquad \bu_i\in X_c^+.
\end{equation}
Repeated words may be deleted by additive idempotence. We shall always
assume that the words in a displayed sum such as \eqref{eq:term} are
distinct. As in \cite{Gao},
\[
 c(\bu)=\bigcup_{i=1}^m c(\bu_i).
\]
We write $|\bu|=m$ for the number of distinct words in $\bu$.

We identify a normalized term with the finite nonempty set of its
words. Under this identification, addition is union and multiplication
is setwise multiplication followed by deletion of repetitions. The free
commutative ai-semiring on $X$ is therefore $\Pf(X_c^+)$, where
$\Pf$ denotes the set of finite nonempty subsets. This model is used
for semilattice-ordered semigroups and commutative ai-semirings; see
\cite{KurilPolak,ShaoRenGao}. The notation $\bw\in\bu$ means that
$\bw$ is one of the words of $\bu$.

A substitution is an endomorphism
\[
 \varphi:\Pf(X_c^+)\longrightarrow\Pf(X_c^+).
\]
It is determined by assigning a finite nonempty sum of words
$\varphi(x)$ to each variable $x$. In a distributive expansion, a
summand of $\varphi(x)$ is chosen independently for each occurrence
of $x$. Every chosen word belongs to $X_c^+$.

We take the commutative ai-semiring laws as background identities.
Finite basability relative to these laws is equivalent to finite
basability in the signature $(+,\cdot)$. Indeed, the background laws
form a finite set, so adjoining them to a finite relative basis gives
a finite identity basis. Conversely, a finite identity basis is also
a finite relative basis.

For an ai-semiring $S$, its natural order is given by
\begin{equation}\label{eq:naturalorder}
 r\leq s\quad\Longleftrightarrow\quad r+s=s.
\end{equation}
Multiplication preserves this order. For example, if $r+s=s$, then
\[
 rt+st=(r+s)t=st,
\]
so $rt\leq st$. The analogous assertion for left multiplication follows
from the other distributive law.

For terms $\bu$ and $\bv$, we use the notation of \cite{Gao}:
\[
 \bu\preceq\bv\quad\text{denotes the identity}\quad
 \bu+\bv\approx\bv.
\]
We write $\bu\preceq_S\bv$ when this identity holds in $S$.
Thus $\bq\preceq_S\bu$ means
\[
 \alpha(\bq)\leq\alpha(\bu)
\]
for every assignment $\alpha$ into $S$.

\begin{lemma}\label{lem:normalization}
Let $\bu=\bu_1+\cdots+\bu_m$ and
$\bv=\bv_1+\cdots+\bv_n$. Relative to the commutative ai-semiring
laws, the identity $\bu\approx\bv$ is equivalent to the finite family
\begin{equation}\label{eq:normalization}
 \bu_i\preceq\bv\quad(1\leq i\leq m),
 \qquad
 \bv_j\preceq\bu\quad(1\leq j\leq n).
\end{equation}
Consequently, a finite identity basis can be replaced by a finite family
of inequalities $\bq\preceq\bu$ in which $\bq$ is a word.
\end{lemma}

\begin{proof}
Suppose first that $\bu\approx\bv$ holds. Each word $\bu_i$ is a
summand of $\bu$, so additive idempotence gives
$\bu_i+\bu\approx\bu$. Replacing $\bu$ by $\bv$ yields
$\bu_i+\bv\approx\bv$, which is $\bu_i\preceq\bv$.
The inequalities $\bv_j\preceq\bu$ follow in the same way.

Conversely, suppose that all the inequalities in
\eqref{eq:normalization} hold. Adding the identities
$\bu_i+\bv\approx\bv$ over $i$ and using additive idempotence gives
\[
 \bu+\bv\approx\bv.
\]
Adding the identities $\bv_j+\bu\approx\bu$ similarly gives
\[
 \bv+\bu\approx\bu.
\]
Since addition is commutative, these two identities imply
$\bu\approx\bv$. Each original identity produces only finitely many
inequalities, so a finite set of identities produces a finite family
of the required form.
\end{proof}

For elementary deductions, we write $\bp=1$ for an absent
multiplicative context and omit the summand $\bd$ when the additive
context is absent.

\begin{lemma}\label{lem:deduction}
Let $\Delta$ be a family of inequalities $\bq\preceq\bu$, each
interpreted as the identity $\bu\approx\bu+\bq$. An elementary
application of an identity in $\Delta$, after normalization by the
commutative ai-semiring laws, has the form
\begin{equation}\label{eq:step}
 \bp\varphi(\bu)+\bd
 \quad\longleftrightarrow\quad
 \bp\varphi(\bu)+\bp\varphi(\bq)+\bd,
\end{equation}
where $\varphi$ is a substitution, $\bp$ is a term or is $1$, and
$\bd$ is a term or is omitted. Every identity deducible from $\Delta$
relative to the commutative ai-semiring laws can be obtained through a
finite sequence of such elementary applications.
\end{lemma}

\begin{proof}
We first verify the assertion about contexts. A context has one
distinguished occurrence, represented by a hole $\square$, into which
a term is to be inserted. We claim that every such context, after
distributive expansion and commutative normalization, can be written
as
\begin{equation}\label{eq:context}
 \bp\,\square+\bd,
\end{equation}
with the stated conventions.

For the context consisting only of $\square$, take $\bp=1$ and omit
$\bd$. Suppose that \eqref{eq:context} has been established for a
context $C[\square]$, and let $\bv$ be a term not containing the
hole. For $C[\square]+\bv$ or $\bv+C[\square]$, the coefficient
$\bp$ is unchanged and the additive remainder becomes $\bd+\bv$;
when $\bd$ was omitted, the new remainder is $\bv$. For
$C[\square]\bv$ or $\bv C[\square]$, distributivity and
commutativity give
\[
 (\bp\,\square+\bd)\bv
 = (\bp\bv)\,\square+\bd\bv.
\]
If $\bp=1$, the new coefficient is $\bv$. If $\bd$ was omitted,
there is no additive remainder in the resulting expression. These are
all the ways to build a context in the language with two binary
operations, so induction proves \eqref{eq:context}. The coefficient
$\bp$ is a term, or is $1$ when the multiplicative context is absent.

Insert $\varphi(\bu)$ and $\varphi(\bu+\bq)$ into the same
context. Since $\varphi$ preserves addition and multiplication,
\[
 \bp\varphi(\bu+\bq)+\bd
 =\bp\varphi(\bu)+\bp\varphi(\bq)+\bd,
\]
which gives \eqref{eq:step}.

For the final assertion, recall that equational deduction is generated
by substitution, replacement in a context, symmetry and transitivity;
see \cite{BurrisSankappanavar,Dolinka}. A replacement at several
occurrences can be performed one occurrence at a time. Symmetry permits
either direction of \eqref{eq:step}, and transitivity concatenates
finitely many steps. The commutative ai-semiring laws have already been
incorporated into the algebra $\Pf(X_c^+)$, so their applications do
not change a normalized term. Thus a deduction gives a finite chain
of the stated kind.
\end{proof}

A left-to-right step of \eqref{eq:step} retains every old word and
adds words from $\bp\varphi(\bq)$. A right-to-left step replaces the
set of words by a subset. Thus a property inherited by subsets is
preserved under every right-to-left step.

\section{The divisibility-order construction}
\label{sec:semiring}

Let $W$ be a nonempty subset of $X_c^+$. As in \cite{Gao}, let
$W^{\leq}$ be the set of all nonempty commutative subwords of words
in $W$. For $u,v\in X_c^+$, their supremum in the divisibility
order is obtained by taking the maximum of the two exponents of each
letter. The semiring $S_c^*(W)$ has universe $W^{\leq}\cup\{0\}$,
with $0$ absorbing for both operations, and
\begin{align}
 u\cdot v&=
 \begin{cases}
 uv,&uv\in W^{\leq},\\
 0,&uv\notin W^{\leq},
 \end{cases}\label{eq:multiplication}\\
 u+v&=
 \begin{cases}
 \sup\{u,v\},&\sup\{u,v\}\in W^{\leq},\\
 0,&\sup\{u,v\}\notin W^{\leq}.
 \end{cases}\label{eq:addition}
\end{align}
This is the ideal quotient construction of \cite[Section~4]{Gao}.
On the free commutative semigroup with an absorbing zero, addition is
coordinatewise maximum of exponent vectors and multiplication adds
exponent vectors. The complement of $W^{\leq}$ is a multiplicative
ideal and an upper set, so identifying this complement gives precisely
\eqref{eq:multiplication}--\eqref{eq:addition}. In particular, the
operations satisfy the ai-semiring laws. When $W=\{w\}$, we write
$S_c^*(w)$.

\subsection{Powers of a single letter}
For $n\geq1$, the elements of $\Spow{n}$ are
\[
 a,a^2,\ldots,a^n,0.
\]
Its natural order is the chain
$a<a^2<\cdots<a^n<0$, and
\begin{equation}\label{eq:poweroperations}
 a^i+a^j=a^{\max\{i,j\}},\qquad
 a^ia^j=
 \begin{cases}
 a^{i+j},&i+j\leq n,\\
 0,&i+j>n.
 \end{cases}
\end{equation}
Thus every word of length at least $n+1$ has constant value $0$.
The map fixing $a,\ldots,a^n,0$ and sending $a^{n+1}$ to $0$
is a surjective homomorphism
\begin{equation}\label{eq:powerquotient}
 \Spow{n+1}\longrightarrow\Spow{n}.
\end{equation}
The unrestricted construction $\Spowinf$ has all positive powers
of $a$ as its nonzero elements. Their multiplication is not truncated.

\subsection{Linear words}
Let $k\geq1$ and let $a_1,\ldots,a_k$ be distinct letters. For
$W=\{a_1\cdots a_k\}$,
\begin{equation}\label{eq:subwords}
 W^{\leq}=
 \{a_{i_1}\cdots a_{i_r}:1\leq r\leq k,
                    \ 1\leq i_1<\cdots<i_r\leq k\}.
\end{equation}
It follows that $|\Slin{k}|=2^k$. Two nonzero elements have
nonzero product precisely when their contents are disjoint. Their sum
is always nonzero and contains every letter occurring in either element,
once each. Explicitly, if
$u=\prod_i a_i^{\varepsilon_i}$ and
$v=\prod_i a_i^{\delta_i}$, with exponents in $\{0,1\}$, then
\begin{equation}\label{eq:sup}
 u+v=\prod_{i=1}^k a_i^{\max\{\varepsilon_i,\delta_i\}}.
\end{equation}
Zero-exponent factors are omitted. For nonzero $u,v$,
\begin{equation}\label{eq:divisibility}
 u\leq v\quad\Longleftrightarrow\quad
 u\text{ is a commutative subword of }v.
\end{equation}
The element $0$ is the greatest element of this order.

The same description applies to $\Slininf$, whose nonzero
elements are all nonempty square-free commutative words on
$\{a_1,a_2,\ldots\}$. In this case,
\begin{equation}\label{eq:linearunion}
 \Slininf=\bigcup_{k\geq1}\Slin{k},
\end{equation}
where the inclusions are the natural subsemiring inclusions.

\begin{lemma}\label{lem:values}
Let $s_1,\ldots,s_m\in\Slin{k}$, where $m\geq1$.
\begin{enumerate}
\item[(i)] $s_1+\cdots+s_m\neq0$ if and only if every $s_i\neq0$.
\item[(ii)] The product $s_1\cdots s_m$ is nonzero if and only if
all factors are nonzero and their contents are pairwise disjoint.
\item[(iii)] If $m=k$ and $s_1\cdots s_k\neq0$, then
$s_i=a_{\pi(i)}$ for a permutation $\pi$ of $[k]$.
\item[(iv)] If $r\leq s$ and $s\neq0$, then $r\neq0$.
\end{enumerate}
Assertions \textup{(i)}, \textup{(ii)} and \textup{(iv)} also hold in
$\Slininf$.
\end{lemma}

\begin{proof}
A zero summand makes a sum zero by absorption. If all summands are
nonzero, their sum is the square-free word whose letters form the union
of their contents. This union is a nonempty finite subset of the
available alphabet, so the sum is nonzero. This proves (i).

A zero factor makes a product zero. If a letter occurs in two factors,
the product is not square-free and is zero by
\eqref{eq:multiplication}. Conversely, factors with disjoint contents
multiply to their square-free concatenation. This proves (ii).

For (iii), $k$ nonzero factors with disjoint contents have total length
at least $k$. A nonzero product in $\Slin{k}$ has length at
most $k$. Every factor therefore has length one, and the $k$ factors
are precisely the $k$ distinct letters.

For (iv), if $r=0$ and $r\leq s$, then $0+s=s$, whereas absorption
gives $0+s=0$. Hence $s=0$, a contradiction. The proofs of (i), (ii)
and (iv) use only finitely many letters and apply unchanged to
$\Slininf$.
\end{proof}

\section{Finite powers of one letter}
\label{sec:powers}

The chain order in \eqref{eq:poweroperations} permits a bounded
selection from the upper term of any valid inequality.

\begin{lemma}\label{lem:powerselection}
Let $n\geq1$ and suppose that $\bq\preceq_{\Spow{n}}\bu$,
where $\bq$ is a word and every word of $\bu$ has length at most
$n$. Then $\ell(\bq)\leq n$ and $c(\bq)\subseteq c(\bu)$.
Moreover, there is a nonempty additive subterm $\bu_0$ of $\bu$
such that
\[
 \bq\preceq_{\Spow{n}}\bu_0,
\]
and the variables of $\bq$ and $\bu_0$ together number at most
\begin{equation}\label{eq:powerbound}
 n\binom{2n+1}{n+1}.
\end{equation}
\end{lemma}

\begin{proof}
Assign $a$ to every variable. Every word of $\bu$ then has a
nonzero value, and so does their sum. If $\ell(\bq)>n$, its value
would be $0$, contradicting the inequality. Hence $\ell(\bq)\leq n$.
If $z\in c(\bq)\setminus c(\bu)$, assign $0$ to $z$ and $a$
to all other variables. The left-hand side is zero and the right-hand
side is nonzero, again a contradiction. Thus
$c(\bq)\subseteq c(\bu)$.

Write $c(\bq)=\{x_1,\ldots,x_r\}$, where $r\leq n$. Associate
to each word $\bw\in\bu$ the vector
\begin{equation}\label{eq:powertype}
 \left(\occ(x_1,\bw),\ldots,\occ(x_r,\bw),
       \sum_{z\notin c(\bq)}\occ(z,\bw)\right).
\end{equation}
Its entries are nonnegative integers with sum between $1$ and $n$.
There are
\[
 \binom{n+r+1}{r+1}-1
\]
possible such vectors. Choose one word from each vector that occurs,
and let $\bu_0$ be the sum of the chosen words. Every variable of
$c(\bq)$ still occurs in $\bu_0$: a positive entry in the
corresponding coordinate remains positive in any chosen representative.

Let $\alpha$ be an arbitrary assignment into $\Spow{n}$. If
$\alpha(\bu_0)=0$, the desired inequality holds because $0$ is
the greatest element. Suppose that $\alpha(\bu_0)\neq0$.
In particular, every variable of $c(\bq)$ has a nonzero value.
Define another assignment $\beta$ by
\[
 \beta(x)=
 \begin{cases}
 \alpha(x),&x\in c(\bq),\\
 a,&x\notin c(\bq).
 \end{cases}
\]
Since $a$ is the least element, $\beta(x)\leq\alpha(x)$ for
every variable. Both operations preserve the order, and therefore
$\beta(\bu_0)\leq\alpha(\bu_0)$.

Under $\beta$, words having the same vector
\eqref{eq:powertype} have the same value: variables outside
$c(\bq)$ all contribute exponent one. By the choice of representatives,
\[
 \beta(\bu)=\beta(\bu_0).
\]
The original inequality now gives
\[
 \alpha(\bq)=\beta(\bq)
 \leq\beta(\bu)=\beta(\bu_0)
 \leq\alpha(\bu_0).
\]
This proves $\bq\preceq_{\Spow{n}}\bu_0$.

Each selected word has length at most $n$. Since
$c(\bq)\subseteq c(\bu_0)$, the total number of variables is
at most
\[
 n\left(\binom{n+r+1}{r+1}-1\right)
 \leq n\binom{2n+1}{n+1},
\]
as required.
\end{proof}

\begin{theorem}\label{thm:finitepowers}
For every integer $n\geq1$, the semiring $\Spow{n}$ is finitely
based.
\end{theorem}

\begin{proof}
In addition to the commutative ai-semiring laws, take the identity
\begin{equation}\label{eq:longabsorption}
 y\preceq x_1\cdots x_{n+1}
\end{equation}
and every inequality $\bq\preceq\bu$ which holds in $\Spow{n}$,
uses variables from a fixed alphabet of size \eqref{eq:powerbound},
and has all its words of length at most $n$.
This is a finite set. Indeed, there are finitely many words of bounded
length on that alphabet and finitely many nonempty subsets of this
finite set of words. All the chosen identities hold in $\Spow{n}$.

Consider an arbitrary valid inequality $\bq\preceq\bu$.
If $\bu$ has a word $\bw$ of length at least $n+1$, write
$\bw$ as a product of $n+1$ nonempty words. Substitute these
words for $x_1,\ldots,x_{n+1}$ in
\eqref{eq:longabsorption}, and substitute $\bq$ for $y$.
This gives $\bq\preceq\bw$, and hence $\bq\preceq\bu$.

Otherwise Lemma~\ref{lem:powerselection} gives $\bu_0$ such that
$\bq\preceq\bu_0$ is valid and uses at most the number of variables
in \eqref{eq:powerbound}. After renaming variables, this inequality
is one of the finitely many selected inequalities. Since $\bu_0$
is an additive subterm of $\bu$, it implies
$\bq\preceq\bu$. Lemma~\ref{lem:normalization} now proves
that every valid identity follows from the proposed finite set.
\end{proof}

\begin{proposition}\label{prop:powerchain}
The varieties generated by finite powers form a strictly ascending chain:
\[
 \V(\Spow{1})\subsetneq\V(\Spow{2})
 \subsetneq\V(\Spow{3})\subsetneq\cdots.
\]
\end{proposition}

\begin{proof}
The quotient maps \eqref{eq:powerquotient} give the inclusions.
The identity $x^{n+1}\approx x^{n+2}$ holds in $\Spow{n}$.
It fails in $\Spow{n+1}$ at $x=a$, where its two values are
$a^{n+1}$ and $0$. Thus each inclusion is strict.
\end{proof}

\section{The unrestricted linear family}
\label{sec:unrestricted}

The unrestricted linear-word semiring admits a normal form described
by finite simple graphs. We first establish its identity basis and then
record the two smallest finite cases.

\begin{theorem}\label{thm:unrestrictedbasis}
Relative to the commutative ai-semiring laws, a finite identity basis
for $\Slininf$ is
\begin{align}
 x&\preceq xy,\label{eq:extensive}\\
 y&\preceq x^2,\label{eq:squaretop}\\
 xyz&\approx xy+xz+yz.\label{eq:triangleexpansion}
\end{align}
\end{theorem}

\begin{proof}
These identities hold in the stated semiring. A product is either zero
or a square-free word containing each factor, which gives
\eqref{eq:extensive}. Every square is zero, giving
\eqref{eq:squaretop}. For \eqref{eq:triangleexpansion}, a repeated
letter in the three factors makes both the triple product and at least
one pair product zero. If no letter is repeated, both sides are the
square-free word containing all letters of the three factors.

We prove completeness. In every model of \eqref{eq:squaretop}, each
square is the greatest element in the additive order, so all squares
coincide. By \eqref{eq:extensive}, this greatest element is also
multiplicatively absorbing: if its value is $s$, then
$s\leq st\leq s$, whence $st=s$. Thus a term containing any
nonlinear word is equivalent to a square term.

Consider a term all of whose words are linear. Induction on $r$, using
\eqref{eq:triangleexpansion}, gives
\begin{equation}\label{eq:pairsexpansion}
 x_1\cdots x_r\approx\sum_{1\leq i<j\leq r}x_ix_j
 \qquad(r\geq2).
\end{equation}
For the induction step, apply \eqref{eq:triangleexpansion} to
$x_1\cdots x_{r-2}$, $x_{r-1}$ and $x_r$. The two resulting
words of length $r-1$ supply all pairs except
$x_{r-1}x_r$, and the third summand supplies that pair.

For a term $\bu$ with only linear words, let $\GG_{\bu}$ have
vertex set $c(\bu)$ and an edge $\{x,y\}$ whenever $x$ and $y$
occur together in a word of $\bu$. Let $I$ be its isolated vertices.
By \eqref{eq:pairsexpansion} and \eqref{eq:extensive}, the term is
equivalent to
\begin{equation}\label{eq:graphnormalform}
 \sum_{\{x,y\}\in E(\GG_{\bu})}xy+\sum_{z\in I}z.
\end{equation}
Empty sums in this display are omitted. At least one summand remains.
A singleton variable occurring in an edge is absorbed by that edge,
which explains why only isolated vertices occur as singleton summands.

A graph normal form is distinct from the square normal form: assign
different letters to all its variables. Its value is a nonzero
square-free word, whereas a square has value zero. Two graph normal
forms with different vertex sets are distinguished by the same
assignment, since their values have different contents.

Finally, suppose that two graph normal forms have the same vertices
but different edge sets. Choose an edge $\{x,y\}$ present in one and
absent from the other. Assign the same letter to $x$ and $y$, and
assign distinct new letters to every other variable. The form containing
$xy$ has value zero. In the other form every pair product is nonzero,
so its value is nonzero by Lemma~\ref{lem:values}(i). Hence these
normal forms are distinguished as well. Every identity holding in
$\Slininf$ therefore has equal normal forms and follows from
\eqref{eq:extensive}--\eqref{eq:triangleexpansion}.
\end{proof}

\Needspace{10\baselineskip}
The proof gives a useful criterion for inequalities.

\begin{lemma}\label{lem:unrestrictedcriterion}
Let $\bq$ be a word and $\bu$ a term. If $\bu$ contains a
nonlinear word, then $\bq\preceq_{\Slininf}\bu$.
If all words of $\bu$ are linear, then
$\bq\preceq_{\Slininf}\bu$ holds if and only if
\begin{enumerate}
\item[(i)] $\bq$ is linear;
\item[(ii)] $c(\bq)\subseteq c(\bu)$;
\item[(iii)] each pair of distinct variables in $c(\bq)$ occurs
together in some word of $\bu$.
\end{enumerate}
\end{lemma}

\begin{proof}
A nonlinear word has constant value zero, so the first assertion
follows from absorption. Suppose that all words of $\bu$ are linear.
Assigning distinct letters to all variables gives a nonzero value to
$\bu$ and shows that a nonlinear $\bq$ is impossible.
If a variable of $\bq$ is absent from $\bu$, assign zero to it
and distinct letters to the other variables; this disproves the
inequality. If $x,y\in c(\bq)$ never occur together in a word
of $\bu$, assign the same letter to $x,y$ and distinct new letters
to all remaining variables. Then $\bu$ is nonzero and $\bq$ is
zero. These assignments prove necessity.

For sufficiency, take an assignment for which $\bu$ is nonzero.
Every variable in its content has a nonzero value. By (iii) and
Lemma~\ref{lem:values}(ii), the values of the distinct variables of
$\bq$ have pairwise disjoint contents. Hence $\bq$ is nonzero.
By (ii), each letter in its value also occurs in the value of $\bu$.
The divisibility order now gives the inequality. If $\bu$ has value
zero, the inequality holds automatically.
\end{proof}

\begin{corollary}\label{cor:linearinfjoin}
One has
\begin{equation}\label{eq:linearinfjoin}
 \V(\Slininf)=\bigvee_{k\geq1}\V(\Slin{k}).
\end{equation}
This variety is locally finite and is not generated by finitely many
finite algebras.
\end{corollary}

\begin{proof}
The union in \eqref{eq:linearunion} is directed. Any evaluation of
finitely many terms uses finitely many elements and therefore occurs
in some $\Slin{k}$. An identity holds in the union exactly when
it holds in every finite member, proving \eqref{eq:linearinfjoin}.

On any fixed finite set of variables there are only finitely many graph
normal forms and the square normal form. Hence the free algebra on
that set is finite, proving local finiteness.

In a finite member of this variety, every square is the common absorbing
greatest element by \eqref{eq:squaretop} and
\eqref{eq:extensive}. If the member has $d$ elements, every product
of $d+1$ elements contains a repeated factor, and hence is this
greatest element. Consequently, any finite collection of finite members
has a common bound on its nilpotence degree. The semirings
$\Slin{k}$ have nonzero products of $k$ factors for arbitrarily
large $k$. They cannot all belong to a variety with such a common bound.
\end{proof}

\begin{proposition}\label{prop:insideS53}
The variety $\V(\Slininf)$ is the subvariety of $\V(S_{53})$
defined by $y\preceq x^2$, and it is a proper subvariety.
\end{proposition}

\begin{proof}
By \cite[Lemma~4.4]{Gao} and
Corollary~\ref{cor:linearinfjoin}, $\V(\Slininf)$ is contained
in $\V(S_{53})$. The identities \eqref{eq:extensive} and
\eqref{eq:triangleexpansion} hold in $S_{53}$; see the basis
recalled in \cite[proof of Theorem~1.2]{ShaoRenGao}.
Thus adjoining $y\preceq x^2$ to the identities of $S_{53}$ gives
all the identities in Theorem~\ref{thm:unrestrictedbasis}, proving
the reverse inclusion. In the notation of \cite[Table~2]{Gao},
$S_{53}=\{0,a,1\}$ with $1<a<0$ and $1^2=1$.
It does not satisfy $y\preceq x^2$, for example at $x=1,y=a$.
The inclusion is therefore strict.
\end{proof}

\Needspace{7\baselineskip}

\begin{proposition}\label{prop:smalllinear}
The semirings $S_c^*(a)$ and $S_c^*(ab)$ are finitely based.
Relative to the commutative ai-semiring laws, a basis for $S_c^*(ab)$ is
\begin{equation}\label{eq:linear2basis}
 x\preceq xy,\qquad y\preceq x^2,\qquad
 t\preceq xyz,\qquad xt\preceq xy+yz+zt.
\end{equation}
\end{proposition}

\begin{proof}
The one-letter case is Theorem~\ref{thm:finitepowers} with $n=1$.
The four-element case is also the semiring $S_{(4,369)}$ in
\cite[Lemma~4.3]{ShaoRenGao}; we give the normal-form argument
for the basis \eqref{eq:linear2basis}.

All four identities hold in $S_c^*(ab)$. For the last one, a nonzero
value of $xy+yz+zt$ requires all three pair products to be nonzero.
The values of $x,y,z,t$ must then be single letters alternating between
$a$ and $b$ along the path. Thus $x$ and $t$ receive different
letters, and $xt$ has the same nonzero value $ab$ as the upper term.
When the upper term is zero, the inequality holds as well.

The first three identities make all squares and all words of length
at least three equal to the common absorbing greatest element.
Every other term reduces to a sum of edge words and isolated variables.
The fourth identity permits adding an edge between the endpoints of
any three-edge walk. Repeated application permits adding an edge
between the endpoints of any odd walk: replace its first three edges
by the permitted shortcut, and proceed by induction on its odd length.

If a graph contains an odd cycle, this operation produces a loop,
that is, a square summand, so the term becomes the greatest element.
If the graph is bipartite, it adds precisely all edges between the two
parts of each connected component and adds no other edges. We obtain
a disjoint union of complete bipartite graphs and isolated vertices.

Such a form has a nonzero value under a proper two-coloring, so it is
distinct from the greatest-element form. If two such forms have different
contents, assign zero to a variable present only in one and properly
two-color the other form; this distinguishes them. If their contents
agree but their completed edge sets differ, choose an edge $\{x,y\}$
present only in the first. In the second, the vertices $x,y$ lie either
in the same part of one component or in different components, allowing
a proper two-coloring with $x,y$ equal in color. Interpreting the colors
as $a,b$ makes the first form zero and the second nonzero.
Thus distinct normal forms are separated, proving completeness.
\end{proof}

\section{Hypergraph methods}
\label{sec:hypergraphmethods}
\subsection{Strong colorings and hyperforests}
\label{sec:hyperforests}

We use the hypergraph notation of \cite[Section~3]{Gao}. A finite
hypergraph is a pair $\HH=(V,E)$, where $E$ is a family of nonempty
subsets of the finite set $V$. The elements of $V$ are its vertices,
and the members of $E$ are its hyperedges. A subset of $V$ is a
\emph{subhyperedge} if it is contained in some hyperedge. A hypergraph
is \emph{$k$-uniform} if every hyperedge has exactly $k$ vertices.
A vertex is isolated if it belongs to no hyperedge.

A \emph{Berge cycle} of length $m\geq 2$ is an alternating sequence
\[
 v_1,e_1,v_2,e_2,\ldots,v_m,e_m,v_1,
\]
in which the vertices $v_1,\ldots,v_m$ are distinct, the hyperedges
$e_1,\ldots,e_m$ are distinct, and
\[
 \{v_i,v_{i+1}\}\subseteq e_i\quad(1\leq i<m),
 \qquad
 \{v_m,v_1\}\subseteq e_m.
\]
The \emph{girth} $g(\HH)$ is the length of a shortest Berge cycle.
If there is no Berge cycle, we call $\HH$ a \emph{hyperforest} and
put $g(\HH)=\infty$. In particular, two distinct hyperedges with
two common vertices give a Berge cycle of length two. Every
hyperforest, and every hypergraph of girth greater than two, therefore
has the property that distinct hyperedges meet in at most one vertex.

A \emph{strong $k$-coloring} of a hypergraph $\HH$ is a mapping
\[
 \kappa:V(\HH)\longrightarrow[k]
\]
which is injective on every hyperedge. This definition also applies
when the hyperedges have different sizes, all at most $k$. If $\HH$
is $k$-uniform, a strong $k$-coloring assigns all $k$ colors to every
hyperedge.

Let $B\subseteq V(\HH)$, where $|B|\leq 2$. A partial coloring
$\lambda:B\to[k]$ is \emph{valid} if any two distinct vertices in
$B$ which form a subhyperedge receive distinct colors. Thus two
vertices which do not lie together in a hyperedge may be assigned the
same color. We say that $\HH$ is \emph{2-robustly strong
$k$-colorable} if every valid partial coloring on at most two vertices
extends to a strong $k$-coloring of $\HH$. For $k=3$, this is the
terminology used in \cite{ZhuoLiangWuZhao}. We shall need the following
statement for every $k\geq 3$.

\begin{lemma}\label{lem:robust}
For every integer $k\geq 3$, every finite $k$-uniform hyperforest is
2-robustly strong $k$-colorable.
\end{lemma}

\begin{proof}
We use induction on the number of hyperedges. A hyperforest with no
hyperedges has no restrictions on its colorings. Any given partial
coloring can be extended by assigning arbitrary colors to the remaining
vertices. This proves the base case, including any isolated vertices.

Suppose that the assertion holds for all finite $k$-uniform
hyperforests with fewer than $m$ hyperedges, and let $\FF$ have $m$
hyperedges, where $m\geq 1$. We first find a hyperedge $e$ which
meets the union of all other hyperedges in at most one vertex.

Consider the incidence bipartite graph of $\FF$. One vertex class is
$V(\FF)$, the other is $E(\FF)$, and a vertex $v$ is adjacent to
a hyperedge vertex $f$ precisely when $v\in f$. A cycle of length
$2s$ in this bipartite graph gives a Berge cycle of length $s$ in
$\FF$, and conversely. Hence the incidence graph is a forest.
Choose a component containing a hyperedge vertex and root it at one
such vertex. Among the hyperedge vertices in that component, choose
$e$ at maximum distance from the root. At most one vertex incident
with $e$ lies on the path from $e$ towards the root. If any other
vertex incident with $e$ belonged to a second hyperedge $f$, the
unique path from the root to $f$ would pass through $e$ and this
vertex. The distance to $f$ would then be two greater than the
distance to $e$, a contradiction. This proves the required property
of $e$.

Let $\lambda:B\to[k]$ be a valid partial coloring, with $|B|\leq 2$.
We show that it extends to all of $\FF$.

First suppose that $e$ is disjoint from every other hyperedge. Delete
$e$ and all its vertices, obtaining a hyperforest $\FF'$ with
$m-1$ hyperedges. The restriction of $\lambda$ to $V(\FF')$ is
valid, so the induction hypothesis extends it to a strong coloring of
$\FF'$. At most two vertices of $e$ were precolored. If two were
precolored, validity of $\lambda$ says that their colors are different,
since they lie in the same hyperedge. Give the uncolored vertices of
$e$ the unused colors, each once. This uses all $k$ colors on $e$
and does not affect any other hyperedge. It also covers the case in
which $e$ is the only hyperedge in its component.

Now suppose that $e$ meets the union of the other hyperedges in
exactly one vertex $r$. Write
\[
 e=\{r\}\cup P,
 \qquad |P|=k-1.
\]
Every vertex of $P$ belongs only to $e$. Delete $e$ and the vertices
of $P$, but retain $r$. The resulting hyperforest $\FF'$ has
$m-1$ hyperedges. We distinguish the following cases according to the
precolored vertices in $P$.

\smallskip
\noindent\emph{Case 1: $B\cap P=\varnothing$.}
All precolored vertices belong to $V(\FF')$. Apply the induction
hypothesis to extend $\lambda$ to a strong coloring of $\FF'$.
Let $c_r$ be the color of $r$ in this extension. Give the $k-1$
vertices of $P$ the $k-1$ colors in $[k]\setminus\{c_r\}$,
each once. The extended coloring is strong on $e$, as well as on
all the hyperedges of $\FF'$.

\smallskip
\noindent\emph{Case 2: $B\cap P=\{p\}$.}
Put $c_p=\lambda(p)$. If $r\in B$, then $\lambda(r)\neq c_p$
by validity. The restriction of $\lambda$ to $V(\FF')$ is valid,
so extend it by induction. In this extension $r$ retains its prescribed
color. Assign the vertices of $P\setminus\{p\}$ the $k-2$
colors different from $c_p$ and $\lambda(r)$, each once.

Suppose instead that $r\notin B$. There is at most one other
precolored vertex, say $z$, and it lies in $V(\FF')$. Choose a
color $c_r$ for $r$ subject to
\[
 c_r\neq c_p,
\]
and, when $z$ exists and $\{r,z\}$ is a subhyperedge of $\FF'$,
also require
\[
 c_r\neq\lambda(z).
\]
At most two colors are excluded. Since $k\geq 3$, such a color
$c_r$ exists. The restriction of $\lambda$ to $V(\FF')$,
together with $r\mapsto c_r$, is a partial coloring on at most two
vertices. It is valid: its only possible two-vertex restriction is on
$\{r,z\}$, and we imposed different colors whenever this pair is a
subhyperedge. Apply the induction hypothesis to this partial coloring.
Finally, assign to $P\setminus\{p\}$ the colors in
$[k]\setminus\{c_r,c_p\}$, each once. This completes $e$ without
changing either prescribed color.

\smallskip
\noindent\emph{Case 3: $|B\cap P|=2$.}
Write $B=\{p_1,p_2\}$. Both vertices lie in $e$, so
$\lambda(p_1)\neq\lambda(p_2)$. No vertex outside $P$ is precolored.
Choose for $r$ a color $c_r$ different from these two colors; again
this is possible because $k\geq 3$. By induction, the one-vertex
partial coloring $r\mapsto c_r$ extends to a strong coloring of
$\FF'$. There remain $k-3$ uncolored vertices in
$P\setminus\{p_1,p_2\}$ and exactly $k-3$ unused colors on $e$.
Assign those colors bijectively to those vertices. For $k=3$, there
are no remaining vertices and this final assignment is empty.

\smallskip
All possible values of $|B\cap P|$ have been covered. In each case
we obtain a strong $k$-coloring of $\FF$ extending $\lambda$.
This completes the induction.
\end{proof}

Taking the empty partial coloring gives a strong $k$-coloring of
every finite $k$-uniform hyperforest. If two vertices form no
subhyperedge, assigning them the same color gives a valid partial
coloring, which also extends to a strong $k$-coloring.

\begin{lemma}\label{lem:clique}
Let $\FF$ be a finite hyperforest such that
\[
 V(\FF)=\bigcup_{e\in E(\FF)}e.
\]
Let $B$ be a nonempty subset of $V(\FF)$. If every pair of distinct
vertices of $B$ is a subhyperedge of $\FF$, then $B$ is itself a
subhyperedge of $\FF$.
\end{lemma}

\begin{proof}
If $|B|=1$, the sole vertex of $B$ belongs to a hyperedge by the
assumption on $V(\FF)$, and the conclusion follows.

Suppose that $|B|\geq 2$. Choose distinct vertices $x,y\in B$ and
a hyperedge $e$ containing $\{x,y\}$. We claim that every vertex
of $B$ belongs to this same $e$. Let $z\in B\setminus\{x,y\}$.
By the pairwise assumption there are hyperedges $f,g$ such that
\[
 \{x,z\}\subseteq f,
 \qquad
 \{y,z\}\subseteq g.
\]
Suppose that $z\notin e$. Then $f\neq e$ and $g\neq e$.
If $f=g$, the two distinct hyperedges $e$ and $f$ both contain
$x,y$, so
\[
 x,e,y,f,x
\]
is a Berge cycle of length two. If $f\neq g$, then $e,g,f$ are
three distinct hyperedges, and
\[
 x,e,y,g,z,f,x
\]
is a Berge cycle of length three. Both possibilities contradict that
$\FF$ is a hyperforest. Therefore $z\in e$.

Since every $z\in B\setminus\{x,y\}$ lies in $e$, and $x,y$
already lie in $e$, we have $B\subseteq e$, as required.
\end{proof}

\subsection{Hypergraph terms and their values}
\label{sec:values}

Let $\HH=(V,E)$ be a finite $k$-uniform hypergraph with
$E\neq\varnothing$, and let $\{x_v\mid v\in V\}$ be variables
in one-to-one correspondence with its vertices. We use the hypergraph
term and the vertex word from \cite[Section~3]{Gao}:
\begin{equation}\label{eq:hypergraphterms}
 \bt_{\HH}
 =\sum_{\{v_1,\ldots,v_k\}\in E}x_{v_1}\cdots x_{v_k},
 \qquad
 \bq_{\HH}=\prod_{v\in V}x_v.
\end{equation}
In the definition of $\bt_{\HH}$, every ordering of the vertices of
a hyperedge is included, as in \cite{Gao}. Thus a hyperedge contributes
$k!$ words before multiplicative commutativity is imposed. In the
word $\bq_{\HH}$, choose any one fixed order of the vertices.

We shall use these same terms in commutative ai-semirings. In
$\Pf(X_c^+)$ the $k!$ words coming from one hyperedge are equal,
and additive idempotence deletes the repetitions. Hence the
commutative normal form of $\bt_{\HH}$ is
\begin{equation}\label{eq:commutativeterm}
 \bt_{\HH}=\sum_{e\in E}\prod_{v\in e}x_v.
\end{equation}
Different hyperedges have different contents, so this normal form has
exactly $|E|$ distinct words. Passing from \eqref{eq:hypergraphterms}
to \eqref{eq:commutativeterm} does not change the value of the term
in any commutative ai-semiring.

Via the correspondence $x_v\leftrightarrow v$, we regard the
content of a word in the variables $x_v$ as a subset of $V$.
Thus $c(\bw)$ is a subhyperedge precisely when all variables of
$\bw$ are vertex variables whose corresponding vertices belong to
a common hyperedge.

\begin{lemma}\label{lem:hypergraphvalue}
Let $k\geq 3$, and let $\HH$ be a
finite $k$-uniform hypergraph with at least one hyperedge and no
isolated vertices. The following conditions are equivalent:
\begin{enumerate}
\item[(i)] There is an assignment $\alpha$ into $\Slin{k}$ such that
$\alpha(\bt_{\HH})\neq 0$.
\item[(ii)] The hypergraph $\HH$ has a strong $k$-coloring.
\end{enumerate}
If $\HH$ is not strongly $k$-colorable, then
\begin{equation}\label{eq:hypergraphinequality}
 \bq_{\HH}\preceq_{\Slin{k}}\bt_{\HH}.
\end{equation}
\end{lemma}

\begin{proof}
Suppose first that $\kappa:V(\HH)\to[k]$ is a strong
$k$-coloring. Define
\[
 \alpha(x_v)=a_{\kappa(v)}\qquad(v\in V(\HH)).
\]
For a hyperedge $e=\{v_1,\ldots,v_k\}$, the values
$\kappa(v_1),\ldots,\kappa(v_k)$ are all different. Since there are
$k$ vertices and $k$ colors, these values are precisely the elements
of $[k]$. Thus
\[
 \alpha(x_{v_1}\cdots x_{v_k})
 =a_{\kappa(v_1)}\cdots a_{\kappa(v_k)}
 =a_1\cdots a_k\neq 0.
\]
Every word of $\bt_{\HH}$ has this same value. Since $E$ is
nonempty, additive idempotence gives
$\alpha(\bt_{\HH})=a_1\cdots a_k\neq 0$. This proves (i).

Conversely, suppose that $\alpha(\bt_{\HH})\neq 0$. By
Lemma~\ref{lem:values}(i), every word associated with a hyperedge has
nonzero value. If $e=\{v_1,\ldots,v_k\}$, then
\[
 \alpha(x_{v_1})\cdots\alpha(x_{v_k})\neq 0.
\]
Lemma~\ref{lem:values}(iii) says that the $k$ factors in this product
are exactly the single letters $a_1,\ldots,a_k$, in some order.
Every vertex belongs to a hyperedge, since there are no isolated
vertices. It follows that for each $v\in V(\HH)$ there is a unique
$i\in[k]$ with $\alpha(x_v)=a_i$. Define $\kappa(v)=i$.
The values of this mapping are different on every hyperedge, so
$\kappa$ is a strong $k$-coloring. This proves (ii).

Finally, if $\HH$ has no strong $k$-coloring, the equivalence shows
that $\alpha(\bt_{\HH})=0$ for every assignment $\alpha$ into
$\Slin{k}$. Because $0$ is absorbing for addition,
\[
 \alpha(\bt_{\HH}+\bq_{\HH})
 =0+\alpha(\bq_{\HH})
 =0
 =\alpha(\bt_{\HH}).
\]
This is the identity abbreviated by
$\bq_{\HH}\preceq_{\Slin{k}}\bt_{\HH}$.
\end{proof}

We next record the restrictions imposed by a valid inequality when
the contents of its right-hand words form a strongly colorable
hypergraph.

\begin{lemma}\label{lem:localvalidity}
Let $k\geq 3$, and suppose that
$\bq\preceq_{\Slin{k}}\bu$, where $\bq$ is a word and
\[
 \bu=\bu_1+\cdots+\bu_m.
\]
Assume that each $\bu_i$ is linear and $\ell(\bu_i)\leq k$.
Define the hypergraph $\KK_{\bu}$ by
\[
 V(\KK_{\bu})=c(\bu),
 \qquad
 E(\KK_{\bu})=\{c(\bu_i)\mid 1\leq i\leq m\}.
\]
If $\KK_{\bu}$ has a strong $k$-coloring, then:
\begin{enumerate}
\item[(i)] $c(\bq)\subseteq c(\bu)$;
\item[(ii)] $\bq$ is linear;
\item[(iii)] $\ell(\bq)\leq k$;
\item[(iv)] every strong $k$-coloring of $\KK_{\bu}$ is injective
on $c(\bq)$.
\end{enumerate}
\end{lemma}

\begin{proof}
Choose a strong $k$-coloring $\lambda:c(\bu)\to[k]$. Define an
assignment on these variables by
\[
 \alpha(x)=a_{\lambda(x)}\qquad(x\in c(\bu)).
\]
Each $\bu_i$ is linear, and $\lambda$ is injective on
$c(\bu_i)$. Thus the factors in $\alpha(\bu_i)$ are distinct
single letters, and $\alpha(\bu_i)\neq 0$ by
\eqref{eq:multiplication}. Lemma~\ref{lem:values}(i) now gives
\begin{equation}\label{eq:semantic-upper}
 \alpha(\bu)\neq 0.
\end{equation}

(i) Suppose that a variable $z$ belongs to $c(\bq)$ but not to
$c(\bu)$. Extend $\alpha$ by setting $\alpha(z)=0$, and assign
arbitrary values to any other variables outside $c(\bu)$. This
extension leaves \eqref{eq:semantic-upper} unchanged. Since $z$
occurs in $\bq$, its zero value forces $\alpha(\bq)=0$. On the
other hand, $\bq\preceq_{\Slin{k}}\bu$ implies
\[
 \alpha(\bq)\leq\alpha(\bu).
\]
By Lemma~\ref{lem:values}(iv), the nonzero right-hand side forces
the left-hand side to be nonzero, a contradiction. Hence every variable
of $\bq$ lies in $c(\bu)$.

(ii) We can now evaluate $\bq$ using only the single-letter values
already assigned on $c(\bu)$. If $\bq$ were not linear, some
variable $x$ would occur at least twice. The letter
$a_{\lambda(x)}$ would then occur at least twice in
$\alpha(\bq)$, so $\alpha(\bq)=0$. This contradicts
\eqref{eq:semantic-upper}, the inequality, and
Lemma~\ref{lem:values}(iv). Thus $\bq$ is linear.

(iii) Suppose that $\ell(\bq)>k$. By (ii), the word $\bq$ then
has more than $k$ distinct variables. The map $\lambda$ has only
$k$ possible colors, so it assigns the same color to two of these
variables. The corresponding letter is repeated in
$\alpha(\bq)$, and again $\alpha(\bq)=0$, giving the same
contradiction. Therefore $\ell(\bq)\leq k$.

(iv) Let $\lambda$ now be any strong $k$-coloring of $\KK_{\bu}$,
and define $\alpha$ from it as above. The argument leading to
\eqref{eq:semantic-upper} applies to every such coloring. If two
variables in $c(\bq)$ had the same color under this $\lambda$,
then $\alpha(\bq)=0$, whereas $\alpha(\bu)\neq 0$. This is
impossible because $\bq\preceq_{\Slin{k}}\bu$. Hence every strong
$k$-coloring is injective on $c(\bq)$.
\end{proof}

\subsection{Preservation of the subhyperedge property}
\label{sec:preservation}

Fix $k\geq 3$ and a finite $k$-uniform hypergraph $\HH$. For a
normalized term $\bt$, consider the property
\[
\begin{split}
 (P_{\HH})\qquad
 &\text{every word $\bw\in\bt$ is linear, and}\\
 &\text{$c(\bw)$ is a subhyperedge of $\HH$.}
\end{split}
\]
Here the contents are viewed as vertex sets under the correspondence
of Subsection~\ref{sec:values}. In particular, all variables of $\bt$
are among $x_v$, $v\in V(\HH)$.

\begin{lemma}[Local preservation]\label{lem:preservation}
Let $k\geq 3$. Let $\Delta$ be a
family of inequalities $\bq\preceq\bu$, each of which holds in $\Slin{k}$,
where $\bq$ is a word. Suppose that $N\geq 1$ and that
$|\bu|\leq N$ for every inequality in $\Delta$. Let $\HH$ be a
finite $k$-uniform hypergraph with
\[
 g(\HH)>N.
\]
If a term $\bt$ satisfies $(P_{\HH})$ and $\bt'$ is obtained from
$\bt$ by one elementary application of an identity in $\Delta$,
then $\bt'$ also satisfies $(P_{\HH})$.
\end{lemma}

\begin{proof}
Choose the inequality $\bq\preceq\bu\in\Delta$ used in this
step. By Lemma~\ref{lem:deduction}, the step is of the form
\eqref{eq:step}. In the right-to-left direction, the words of the
resulting term form a subset of the words of the initial term.
Linearity and containment of each content in a hyperedge are therefore
preserved. We need only consider a left-to-right step. Write it as
\begin{equation}\label{eq:forward}
 \begin{aligned}
  \bt&=\bp\varphi(\bu)+\bd,\\
  \bt'&=\bp\varphi(\bu)+\bp\varphi(\bq)+\bd.
 \end{aligned}
\end{equation}
The words already in $\bt$ satisfy the required property. It remains
to prove it for every word of $\bp\varphi(\bq)$. We prove this in six steps.

\Needspace{8\baselineskip}
\smallskip
\noindent\emph{Step 1. Linearity and length of the source words.}

If $\bp$ is a term, choose a word $\bp_0\in\bp$. If $\bp=1$,
put $\bp_0=1$. For each $x\in c(\bu)$, choose one word
$\bv_x^{(0)}\in\varphi(x)$. Each $\bv_x^{(0)}$ is nonempty, because a
substitution replaces variables by nonempty terms whose words belong
to $X_c^+$.

Take a word $\bw\in\bu$ and write
\[
 \bw=x_1^{m_1}\cdots x_s^{m_s},
 \qquad m_i\geq 1,
\]
with $x_1,\ldots,x_s$ distinct. When $\varphi(\bw)$ is expanded,
we may choose $\bv_{x_i}^{(0)}$ for every one of the $m_i$ occurrences
of $x_i$. Therefore the word
\begin{equation}\label{eq:initialword}
 \bp_0(\bv_{x_1}^{(0)})^{m_1}\cdots(\bv_{x_s}^{(0)})^{m_s}
\end{equation}
belongs to $\bp\varphi(\bw)$, and hence to $\bt$.

By $(P_{\HH})$, the word \eqref{eq:initialword} is linear. If
$m_i\geq 2$ for some $i$, choose a variable occurring in the nonempty
word $\bv_{x_i}^{(0)}$. It would occur at least twice in
$(\bv_{x_i}^{(0)})^{m_i}$, and hence at least twice in
\eqref{eq:initialword}. This is impossible. Thus $m_i=1$ for all
$i$, so $\bw$ is linear.

Every factor $\bv_{x_i}^{(0)}$ has length at least one. Consequently,
\[
 \ell\bigl(\bp_0\bv_{x_1}^{(0)}\cdots\bv_{x_s}^{(0)}\bigr)
 =\ell(\bp_0)+\sum_{i=1}^s\ell(\bv_{x_i}^{(0)})
 \geq s=\ell(\bw).
\]
The word on the left is linear and its content is contained in a
hyperedge of $\HH$, which has $k$ vertices. Its length is therefore
at most $k$. It follows that $\ell(\bw)\leq k$. Since
$\bw\in\bu$ was arbitrary, Step~1 is proved.

\Needspace{8\baselineskip}
\smallskip
\noindent\emph{Step 2. The first local hyperforest.}

Retain $\bp_0$ and the words $\bv_x^{(0)}$. By Step~1, each word
$\bw\in\bu$ is linear, so the word
\[
 \bp_0\prod_{x\in c(\bw)}\bv_x^{(0)}
\]
is one of the words of $\bt$. Choose a hyperedge
$e_{\bw}^{(0)}\in E(\HH)$ for which
\begin{equation}\label{eq:initialhosts}
 c\left(\bp_0\prod_{x\in c(\bw)}\bv_x^{(0)}\right)
 \subseteq e_{\bw}^{(0)}.
\end{equation}
Such a hyperedge exists by $(P_{\HH})$. Define $\FF_0$ by
\[
 E(\FF_0)=\{e_{\bw}^{(0)}\mid\bw\in\bu\},
 \qquad
 V(\FF_0)=\bigcup_{\bw\in\bu}e_{\bw}^{(0)}.
\]
Some of the chosen hyperedges may coincide, so
\begin{equation}\label{eq:F0size}
 |E(\FF_0)|\leq|\bu|\leq N.
\end{equation}
All hyperedges of $\FF_0$ are hyperedges of $\HH$ and hence have
size $k$.

Suppose that $\FF_0$ had a Berge cycle of length $m$. A Berge cycle
uses $m$ distinct hyperedges, so
\[
 m\leq|E(\FF_0)|\leq N.
\]
The same alternating sequence would be a Berge cycle in $\HH$,
contradicting $g(\HH)>N$. Therefore $\FF_0$ is a $k$-uniform
hyperforest. Its vertex set is the union of its hyperedges, by
construction.

\Needspace{8\baselineskip}
\smallskip
\noindent\emph{Step 3. A strong coloring of the source hypergraph.}

Let $\KK_{\bu}$ be the hypergraph from
Lemma~\ref{lem:localvalidity}, with vertex set $c(\bu)$ and
hyperedges $c(\bw)$ for $\bw\in\bu$. Step~1 verifies the
linearity and length hypotheses of that lemma. We now construct a
strong $k$-coloring of $\KK_{\bu}$.

For each $x\in c(\bu)$, choose a variable $x_{v_x}$ occurring
in $\bv_x^{(0)}$. This is possible because $\bv_x^{(0)}$ is nonempty.
The index $v_x$ is a vertex of $\FF_0$: the variable $x$ occurs
in some word $\bw\in\bu$, and \eqref{eq:initialhosts} places
all variables of $\bv_x^{(0)}$ in $e_{\bw}^{(0)}$.

If $\bw=x_1\cdots x_s\in\bu$, then
\[
 v_{x_1},\ldots,v_{x_s}\in e_{\bw}^{(0)}.
\]
These vertices are pairwise distinct. Indeed, if $v_{x_i}=v_{x_j}$
for $i\neq j$, the same target variable would occur in both
$\bv_{x_i}^{(0)}$ and $\bv_{x_j}^{(0)}$. It would then be repeated in
$\bp_0\bv_{x_1}^{(0)}\cdots\bv_{x_s}^{(0)}$, contrary to the
linearity of this word in $\bt$.

By Lemma~\ref{lem:robust}, the hyperforest $\FF_0$ has a strong
$k$-coloring $\kappa$. Define
\[
 \lambda(x)=\kappa(v_x)\qquad(x\in c(\bu)).
\]
For a word $\bw=x_1\cdots x_s\in\bu$, the vertices
$v_{x_1},\ldots,v_{x_s}$ are distinct and lie in the single
hyperedge $e_{\bw}^{(0)}$. Since $\kappa$ is a strong coloring, these
vertices receive distinct colors. Thus $\lambda$ is injective on
$c(\bw)$ for every $\bw\in\bu$, which means that $\lambda$ is
a strong $k$-coloring of $\KK_{\bu}$.

The inequality $\bq\preceq\bu$ holds in $\Slin{k}$, so
Lemma~\ref{lem:localvalidity} applies. We obtain
\begin{equation}\label{eq:qform}
 c(\bq)\subseteq c(\bu),
 \qquad
 \bq\text{ is linear},
 \qquad
 1\leq\ell(\bq)\leq k.
\end{equation}
The lower bound follows because $\bq$ is a nonempty word.
In particular, there are distinct variables $x_1,\ldots,x_s$ such that
\[
 \bq=x_1x_2\cdots x_s,
 \qquad 1\leq s\leq k.
\]

\Needspace{8\baselineskip}
\smallskip
\noindent\emph{Step 4. The local hyperforest for a fixed new word.}

Let $\br$ be an arbitrary word of $\bp\varphi(\bq)$. Since
$\bq=x_1\cdots x_s$ is linear, the expansion giving $\br$ selects
one word $\bv_{x_i}\in\varphi(x_i)$ for each $i$. If $\bp$ is
a term, it also selects a word $\bp_1\in\bp$; if $\bp=1$, put
$\bp_1=1$. Thus
\begin{equation}\label{eq:newword}
 \br=\bp_1\bv_{x_1}\cdots\bv_{x_s}.
\end{equation}
By \eqref{eq:qform}, every $x_i$ belongs to $c(\bu)$. For each
$x\in c(\bu)\setminus c(\bq)$, choose an arbitrary word
$\bv_x\in\varphi(x)$. Together with the choices already determined
by $\br$, this fixes a word $\bv_x$ for every $x\in c(\bu)$.
Because $\bq$ is linear, the expansion of $\br$ selects exactly
one word for each variable in $c(\bq)$.

For each $\bw\in\bu$, define
\begin{equation}\label{eq:expandedword}
 \br_{\bw}=\bp_1\prod_{x\in c(\bw)}\bv_x.
\end{equation}
Step~1 shows that $\bw$ is linear. Hence \eqref{eq:expandedword}
is a word in the distributive expansion of $\bp\varphi(\bw)$.
In the finite-set interpretation of terms,
\[
 \br_{\bw}\in\bp\varphi(\bw)
 \subseteq\bp\varphi(\bu)
 \subseteq\bt.
\]
Therefore $\br_{\bw}$ is linear and there is a hyperedge
$e_{\bw}\in E(\HH)$ such that
\begin{equation}\label{eq:hosts}
 c(\br_{\bw})\subseteq e_{\bw}.
\end{equation}
Put
\[
 E(\FF)=\{e_{\bw}\mid\bw\in\bu\},
 \qquad
 V(\FF)=\bigcup_{\bw\in\bu}e_{\bw}.
\]
As in Step~2, repeated choices of a hyperedge are counted only once,
so
\begin{equation}\label{eq:Fsize}
 |E(\FF)|\leq|\bu|\leq N<g(\HH).
\end{equation}
A Berge cycle in $\FF$ would use at most $N$ hyperedges and would
also be a cycle in $\HH$. It follows that $\FF$ is a
$k$-uniform hyperforest.

The word $\bp_1$ occurs as a factor of every $\br_{\bw}$. Moreover,
for each $x\in c(\bu)$, some word $\bw\in\bu$ contains $x$,
and then $\bv_x$ is a factor of $\br_{\bw}$. Formula
\eqref{eq:hosts} consequently gives
\begin{equation}\label{eq:localvariables}
 c(\bp_1)\ \cup\!\bigcup_{x\in c(\bu)}c(\bv_x)
 \subseteq V(\FF),
\end{equation}
under the vertex-variable identification. When $\bp_1=1$, its content
is empty. Since $c(\bq)\subseteq c(\bu)$, all variables of
$\br$ also lie in $V(\FF)$.

\Needspace{8\baselineskip}
\smallskip
\noindent\emph{Step 5. Nonzero values under strong colorings.}

Take an arbitrary strong $k$-coloring
\[
 \kappa:V(\FF)\longrightarrow[k].
\]
Interpret its colors as the letters $a_1,\ldots,a_k$ by setting
\begin{equation}\label{eq:targetassignment}
 \widehat{\kappa}(x_v)=a_{\kappa(v)}
 \qquad(v\in V(\FF)).
\end{equation}
This assignment may be extended arbitrarily to the other variables of
$X$. By \eqref{eq:localvariables}, its values on $V(\FF)$ already
suffice to evaluate every word used below.

For every $\bw\in\bu$, the word $\br_{\bw}$ is linear and has
content contained in $e_{\bw}$. Since $\kappa$ is injective on
$e_{\bw}$, the variables in $\br_{\bw}$ receive pairwise distinct
single letters. Thus
\begin{equation}\label{eq:expandednonzero}
 \widehat{\kappa}(\br_{\bw})\neq 0
 \qquad(\bw\in\bu).
\end{equation}

Define an assignment on the source variables by
\begin{equation}\label{eq:sourceassignment}
 \alpha_\kappa(x)=\widehat{\kappa}(\bv_x)
 \qquad(x\in c(\bu)),
\end{equation}
and extend it arbitrarily to $X$. Each right-hand side is an element
of $\Slin{k}$. The inclusion $c(\bq)\subseteq c(\bu)$ ensures that this
assignment specifies the values of both $\bu$ and $\bq$.

Suppose first that $\bp_1\neq 1$, and put
\[
 \beta_\kappa=\widehat{\kappa}(\bp_1).
\]
For $\bw\in\bu$, multiplicativity and
\eqref{eq:expandedword} give
\[
 \beta_\kappa\alpha_\kappa(\bw)
 =\widehat{\kappa}(\bp_1)
   \prod_{x\in c(\bw)}\widehat{\kappa}(\bv_x)
 =\widehat{\kappa}(\br_{\bw})\neq 0.
\]
Using distributivity and Lemma~\ref{lem:values}(i), we obtain
\begin{equation}\label{eq:upperwithcontext}
 \begin{aligned}
 \beta_\kappa\alpha_\kappa(\bu)
 &=\sum_{\bw\in\bu}\beta_\kappa\alpha_\kappa(\bw)\\
 &=\sum_{\bw\in\bu}\widehat{\kappa}(\br_{\bw})
 \neq 0.
 \end{aligned}
\end{equation}
The inequality $\bq\preceq_{\Slin{k}}\bu$ gives
\[
 \alpha_\kappa(\bq)\leq\alpha_\kappa(\bu).
\]
Since multiplication preserves the natural order, it follows that
\begin{equation}\label{eq:contextinequality}
 \beta_\kappa\alpha_\kappa(\bq)
 \leq\beta_\kappa\alpha_\kappa(\bu).
\end{equation}
The right-hand side is nonzero by \eqref{eq:upperwithcontext}.
Lemma~\ref{lem:values}(iv) therefore implies that the left-hand side
is nonzero. By \eqref{eq:newword} and
\eqref{eq:sourceassignment}, this left-hand side is exactly
\[
 \beta_\kappa\alpha_\kappa(\bq)
 =\widehat{\kappa}(\bp_1)
   \prod_{i=1}^s\widehat{\kappa}(\bv_{x_i})
 =\widehat{\kappa}(\br).
\]
Hence $\widehat{\kappa}(\br)\neq 0$ when $\bp_1\neq 1$.

If $\bp_1=1$, the multiplicative context is absent, and
\eqref{eq:expandedword} and \eqref{eq:sourceassignment} give
\[
 \alpha_\kappa(\bw)=\widehat{\kappa}(\br_{\bw})\neq 0
 \qquad(\bw\in\bu).
\]
Lemma~\ref{lem:values}(i) yields $\alpha_\kappa(\bu)\neq 0$.
The inequality
\[
 \alpha_\kappa(\bq)\leq\alpha_\kappa(\bu)
\]
and Lemma~\ref{lem:values}(iv) imply
$\alpha_\kappa(\bq)\neq 0$. In the present case,
\eqref{eq:newword} says
$\br=\bv_{x_1}\cdots\bv_{x_s}$, so
$\widehat{\kappa}(\br)=\alpha_\kappa(\bq)\neq 0$.

The coloring $\kappa$ was arbitrary. We have proved
\begin{equation}\label{eq:everycoloring}
 \widehat{\kappa}(\br)\neq 0
 \quad\text{for every strong $k$-coloring $\kappa$ of $\FF$}.
\end{equation}

\Needspace{8\baselineskip}
\smallskip
\noindent\emph{Step 6. Linearity and containment in a hyperedge.}

By Lemma~\ref{lem:robust}, the hyperforest $\FF$ has at least one
strong $k$-coloring. If a variable $x_v$ occurred twice in $\br$,
then under every assignment \eqref{eq:targetassignment} the letter
$a_{\kappa(v)}$ would occur twice in its value. Its value would
therefore be zero by \eqref{eq:multiplication}, contrary to
\eqref{eq:everycoloring}. Thus $\br$ is linear.

If $\ell(\br)>k$, then $\br$ contains more than $k$ distinct
variables. Any strong $k$-coloring of $\FF$ uses at most $k$
colors in total, so two of these variables receive the same color.
Their assigned letter is then repeated in the value of $\br$, again
contradicting \eqref{eq:everycoloring}. Hence
\begin{equation}\label{eq:Rlength}
 \ell(\br)\leq k.
\end{equation}
To prove containment in a hyperedge, suppose that two distinct vertices
$u,v\in c(\br)$ do not form a subhyperedge of $\FF$. Give both vertices color $1$. This is a
valid partial coloring on two vertices, because no hyperedge contains
both $u$ and $v$. Lemma~\ref{lem:robust} extends it to a strong
$k$-coloring $\kappa$ of $\FF$. Since both $x_u$ and $x_v$
occur in $\br$, the value $\widehat{\kappa}(\br)$ contains the
letter $a_1$ at least twice and is therefore zero. This contradicts
\eqref{eq:everycoloring}.

We conclude that every pair of distinct vertices of $c(\br)$ is a
subhyperedge of $\FF$. The word $\br$ is nonempty, because
$\bq$ and all the substituted words are nonempty. Also
$c(\br)\subseteq V(\FF)$ by \eqref{eq:localvariables}, and
$V(\FF)$ is the union of its hyperedges. Lemma~\ref{lem:clique}
therefore applies and gives a hyperedge $e\in E(\FF)$ with
\[
 c(\br)\subseteq e.
\]
Every hyperedge of $\FF$ is a hyperedge of $\HH$, so $c(\br)$
is also a subhyperedge of $\HH$.

\smallskip
The word $\br$ was an arbitrary word of $\bp\varphi(\bq)$.
We have shown that each such word is linear and has content contained
in a hyperedge of $\HH$. The other words of $\bt'$ already
belong to $\bt$ and satisfy $(P_{\HH})$ by hypothesis. Thus
\eqref{eq:forward} proves that $\bt'$ satisfies $(P_{\HH})$.
\end{proof}

Each local hyperforest is obtained by choosing one containing
hyperedge for each word of $\bu$. Its number of hyperedges is
therefore bounded by $|\bu|$. The first hyperforest yields
\eqref{eq:qform}, and the second is determined by the chosen word
of $\bp\varphi(\bq)$.

\section{Nonfinite bases for finite linear words}
\label{sec:mainproof}

The ordinary chromatic number $\chi(\HH)$ of a hypergraph is the
least number of colors in a coloring of its vertices with no
monochromatic hyperedge. For hyperedges of size at least two, every
strong $k$-coloring is an ordinary proper $k$-coloring. We use the following theorem of Erd\H{o}s and
Hajnal; see \cite{ErdosHajnal} and the formulation in
\cite{AxenovichKarrer}.

\begin{theorem}[Erd\H{o}s--Hajnal]\label{thm:EH}
For all integers $r,g,h\geq 2$, there exists a finite $r$-uniform
hypergraph $\HH$ whose Berge girth is at least $g$ and whose ordinary
chromatic number is at least $h$.
\end{theorem}

\begin{theorem}\label{thm:main}
For every integer $k\geq3$, the semiring $\Slin{k}$ is nonfinitely based.
\end{theorem}

\begin{proof}
Fix $k\geq 3$. Suppose, for a contradiction, that
$\Slin{k}$ is finitely based. Since it is a commutative
ai-semiring, it has a finite basis relative to the commutative
ai-semiring laws. By Lemma~\ref{lem:normalization}, this basis may be
replaced by a finite family
\begin{equation}\label{eq:finitebasis}
 \Delta=\{\bq_i\preceq\bu_i\mid 1\leq i\leq t\},
\end{equation}
where each $\bq_i$ is a word, each $\bu_i$ is a term, and every
inequality holds in $\Slin{k}$. Define
\begin{equation}\label{eq:N}
 N=\max\bigl(\{1\}\cup\{|\bu_i|\mid 1\leq i\leq t\}\bigr).
\end{equation}
Thus every right-hand term in \eqref{eq:finitebasis} contains at
most $N$ distinct commutative words, and $N\geq 1$.

Apply Theorem~\ref{thm:EH} with
\[
 r=k,\qquad g=\max\{N,3\}+1,\qquad h=k+1.
\]
There is a finite $k$-uniform hypergraph $\HH$ such that
\begin{equation}\label{eq:highgirth}
 g(\HH)>\max\{N,3\},
 \qquad
 \chi(\HH)\geq k+1.
\end{equation}
This hypergraph has at least one hyperedge, since an edgeless
hypergraph can be colored with one color. Delete any isolated vertices.
This does not change the hyperedges and hence does not change the
Berge cycles or the girth. It does not change the chromatic number
either: a proper coloring restricts to the remaining vertices, and
any coloring of the remaining vertices extends to isolated vertices
by assigning them arbitrary colors. We may therefore assume that
$\HH$ has no isolated vertices, while retaining
\eqref{eq:highgirth}.

The hypergraph $\HH$ has no strong $k$-coloring. Indeed, any such
coloring would give $k$ distinct colors on every hyperedge, and in
particular would give an ordinary proper $k$-coloring. This would
contradict $\chi(\HH)\geq k+1$. By
Lemma~\ref{lem:hypergraphvalue}, the following inequality holds in $\Slin{k}$:
\[
 \bq_{\HH}\preceq_{\Slin{k}}\bt_{\HH}.
\]
Equivalently, $\Slin{k}$ satisfies the identity
\begin{equation}\label{eq:targetidentity}
 \bt_{\HH}\approx\bt_{\HH}+\bq_{\HH}.
\end{equation}

Because \eqref{eq:finitebasis} is assumed to be a basis, the identity
\eqref{eq:targetidentity} must follow from $\Delta$ and the
commutative ai-semiring laws. Lemma~\ref{lem:deduction} then gives a
finite chain of normalized terms
\begin{equation}\label{eq:chain}
 \bt_1,\bt_2,\ldots,\bt_m,
 \qquad
 \bt_1=\bt_{\HH},
 \qquad
 \bt_m=\bt_{\HH}+\bq_{\HH},
\end{equation}
in which each pair of successive terms is related by one elementary
application of an identity in $\Delta$, in either direction.

The first term satisfies $(P_{\HH})$. In fact, by
\eqref{eq:commutativeterm}, every word of $\bt_{\HH}$ is of the
form
\[
 \prod_{v\in e}x_v\qquad(e\in E(\HH)).
\]
Each vertex of $e$ occurs exactly once, so this word is linear,
and its content is exactly the hyperedge $e$.

Now suppose that $\bt_j$ satisfies $(P_{\HH})$ for some $j<m$.
The identity used in the step from $\bt_j$ to $\bt_{j+1}$ belongs
to $\Delta$, is valid in $\Slin{k}$, and has at most $N$ words in its
right-hand term by \eqref{eq:N}. Also $g(\HH)>N$ by
\eqref{eq:highgirth}. All hypotheses of
Lemma~\ref{lem:preservation} are satisfied, so $\bt_{j+1}$
satisfies $(P_{\HH})$. Induction along \eqref{eq:chain} gives
\begin{equation}\label{eq:chainproperty}
 \bt_j\text{ satisfies }(P_{\HH})
 \qquad(1\leq j\leq m).
\end{equation}

We finally examine the word $\bq_{\HH}$. By definition,
\[
 \bq_{\HH}=\prod_{v\in V(\HH)}x_v,
 \qquad
 c(\bq_{\HH})=V(\HH)
\]
under the vertex-variable correspondence. We have $|V(\HH)|>k$.
Otherwise, we could assign distinct colors to all vertices using at
most $k$ colors. Since every hyperedge has $k\geq 3$ vertices,
this would be an ordinary proper $k$-coloring, contrary to
\eqref{eq:highgirth}.

Every hyperedge of $\HH$ has exactly $k$ vertices. No hyperedge
can therefore contain $V(\HH)$. Hence $c(\bq_{\HH})$ is not
a subhyperedge of $\HH$. The word $\bq_{\HH}$ has length
$|V(\HH)|>k$, whereas every word in $\bt_{\HH}$ has length
$k$, so $\bq_{\HH}$ is a distinct word in the normalized sum
$\bt_{\HH}+\bq_{\HH}$. This last term fails $(P_{\HH})$,
contradicting \eqref{eq:chainproperty} for $j=m$.

Thus $\Slin{k}$ has no finite basis relative to the commutative ai-semiring
laws. By the equivalence in Section~\ref{sec:preliminaries},
$S_c^*(a_1\cdots a_k)$ is nonfinitely based for every $k\geq 3$.
\end{proof}

\begin{corollary}\label{cor:abc}
The eight-element ai-semiring $S_c^*(abc)$ is nonfinitely based.
\end{corollary}

\begin{proof}
Take $k=3$ in Theorem~\ref{thm:main}. The order of the semiring is
$2^3=8$ by \eqref{eq:subwords}.
\end{proof}

\begin{corollary}\label{cor:independence}
Let $k\geq 3$ and let
$\Sigma\subseteq\Id(S_c^*(a_1\cdots a_k))$ be finite. There is a
finite $k$-uniform hypergraph $\HH$ such that
$\bt_{\HH}\approx\bt_{\HH}+\bq_{\HH}$ holds in
$S_c^*(a_1\cdots a_k)$ and is independent of $\Sigma$ relative
to the commutative ai-semiring laws.
\end{corollary}

\begin{proof}
Normalize $\Sigma$ by Lemma~\ref{lem:normalization} and let $N\geq 1$
bound the number of words in the right-hand terms of the resulting
inequalities. Choose $\HH$ as in \eqref{eq:highgirth}.
Lemma~\ref{lem:hypergraphvalue} gives the stated identity, while
Lemma~\ref{lem:preservation} preserves $(P_{\HH})$ along every
deduction from $\Sigma$. The term $\bt_{\HH}$ satisfies this
property and $\bt_{\HH}+\bq_{\HH}$ fails it. Hence the identity
is independent of $\Sigma$.
\end{proof}

\section{Containment and finite joins}
\label{sec:finitejoins}

We first determine exactly when a variety from the linear-word family
is contained in a variety from the power family.

\begin{lemma}\label{lem:thresholdinequality}
Let $k\geq2$ and $1\leq n\leq2k-2$. Then
\begin{equation}\label{eq:thresholdinequality}
 x_1^2\preceq_{\Spow{n}}x_1x_2\cdots x_k.
\end{equation}
This inequality does not hold in $\Slin{k}$.
\end{lemma}

\begin{proof}
Take an assignment into $\Spow{n}$. If the product on the right
is zero, the inequality holds because zero is greatest. Otherwise all
its factors are nonzero, say $x_i\mapsto a^{r_i}$, with
\[
 r_1+\cdots+r_k\leq n\leq2k-2,\qquad r_i\geq1.
\]
Since $r_2+\cdots+r_k\geq k-1$, we have
\[
 r_1\leq n-(k-1)\leq k-1\leq r_2+\cdots+r_k.
\]
It follows that $2r_1\leq r_1+\cdots+r_k\leq n$. Both sides
are nonzero, and the exponent of the left-hand side is no greater
than that of the right-hand side. This proves
\eqref{eq:thresholdinequality}.

In $\Slin{k}$, assign $a_i$ to $x_i$. The two sides then have
values $0$ and $a_1\cdots a_k$, respectively. Since the latter
is nonzero, the inequality fails.
\end{proof}

\begin{theorem}\label{thm:containment}
For all integers $n,k\geq1$,
\begin{equation}\label{eq:exactcontainment}
 \V(\Slin{k})\subseteq\V(\Spow{n})
 \quad\Longleftrightarrow\quad n\geq2k-1.
\end{equation}
Moreover,
\begin{equation}\label{eq:reversecontainment}
 \V(\Spow{n})\subseteq\V(\Slin{k})
 \quad\Longleftrightarrow\quad n=1.
\end{equation}
\end{theorem}

\begin{proof}
When $n<2k-1$, necessarily $k\geq2$, and
Lemma~\ref{lem:thresholdinequality} rules out the first inclusion.
For its converse, we construct a surjective homomorphism onto
$\Slin{k}$ from a subsemiring of
\[
 \bigl(\Spow{2k-1}\bigr)^k.
\]
Let $A$ be the subsemiring generated by $\alpha_1,\ldots,\alpha_k$,
where the $j$th coordinate of $\alpha_i$ is $a^k$ if $j=i$ and
is $a$ otherwise. Thus
\begin{equation}\label{eq:thresholdgenerators}
 \alpha_i=(a,\ldots,a,a^k,a,\ldots,a),
\end{equation}
with $a^k$ in coordinate $i$.

Every element of $A$ is the value of a nonempty sum of words in these
generators. If such a word contains a repeated generator $\alpha_i$,
its $i$th coordinate has exponent at least $2k$ and therefore equals
zero. The sum containing that word also has a zero coordinate.

For a square-free word with generator set $J\subseteq[k]$ and
$r=|J|\geq1$, direct multiplication gives
\begin{equation}\label{eq:thresholdcoordinates}
 \left(\prod_{i\in J}\alpha_i\right)_j=
 \begin{cases}
 a^{k+r-1},&j\in J,\\
 a^r,&j\notin J.
 \end{cases}
\end{equation}
Here $k+r-1\leq2k-1$. If $j\notin J$, then $r\leq k-1$.
Thus the value has no zero coordinate, and its $j$th exponent is at
least $k$ exactly when $j\in J$.

A sum of such square-free words still has no zero coordinate, since
coordinatewise addition takes maxima. Its $j$th exponent is at least
$k$ exactly when generator $\alpha_j$ occurs in the content of
at least one summand. In particular, at least one exponent is at least
$k$. We may therefore define $h:A\to\Slin{k}$ by
\begin{equation}\label{eq:thresholdmap}
 h(s_1,\ldots,s_k)=
 \begin{cases}
 0,&s_i=0\text{ for some }i,\\[1mm]
 \displaystyle\prod_{\{i:r_i\geq k\}}a_i,
   &s_i=a^{r_i}\text{ for all }i.
 \end{cases}
\end{equation}
The product in the second case is nonempty by the preceding argument.

For any term in the generators, the first case of
\eqref{eq:thresholdmap} occurs exactly when that term has a nonlinear
word. Evaluating the same term at $a_1,\ldots,a_k$ in $\Slin{k}$
also gives zero in exactly this case. If all its words are linear,
\eqref{eq:thresholdcoordinates} shows that $h$ sends its value to
the square-free word given by the union of their contents. This is
again its value at $a_1,\ldots,a_k$.

It follows that $h$ preserves both operations: represent two elements
of $A$ by terms, and apply the preceding observation to their sum
and their product. Also $h(\alpha_i)=a_i$, so $h$ is surjective.
This proves $\Slin{k}\in\V(\Spow{2k-1})$. For every
$n\geq2k-1$, Proposition~\ref{prop:powerchain} gives the required
inclusion into $\V(\Spow{n})$.

For \eqref{eq:reversecontainment}, the case $n=1$ follows from the
subsemiring $\{0,a_1\}\cong S_c^*(a)$ of $\Slin{k}$.
For $n\geq2$, all squares coincide in $\Slin{k}$, whereas
$x^2\approx y^2$ fails in $\Spow{n}$ at $x=a,y=0$.
This rules out the reverse inclusion.
\end{proof}

The only small mixed join not immediately absorbed by
Theorem~\ref{thm:containment} is the join generated by $S_c^*(a^2)$
and $S_c^*(ab)$. Its finite basis can also be obtained by graph normal
forms.

\begin{lemma}\label{lem:smallmixed}
Relative to the commutative ai-semiring laws, the variety
\[
 \V(S_c^*(a^2))\vee\V(S_c^*(ab))
\]
is defined by
\begin{align}
 x&\preceq xy,\label{eq:smallmix1}\\
 t&\preceq xyz,\label{eq:smallmix2}\\
 xy&\preceq x^2+y^2,\label{eq:smallmix3}\\
 xt&\preceq xy+yz+zt,\label{eq:smallmix4}\\
 z^2&\preceq x^2+yz.\label{eq:smallmix5}
\end{align}
In particular, this join is finitely based.
\end{lemma}

\begin{proof}
All five inequalities hold in the two generating semirings. The first
two follow from their divisibility orders and the vanishing of triple
products. In $S_c^*(ab)$, the third and fifth have a square on the
right and hence hold; the fourth was verified in
Proposition~\ref{prop:smalllinear}. In $S_c^*(a^2)$, any nonzero
pair product forces both factors to equal $a$ and has value $a^2$.
For \eqref{eq:smallmix3}, nonzero values of both squares on the right
force both variables to equal $a$. For
\eqref{eq:smallmix4}, nonzero values of the three pair products
force all four variables to equal $a$. For
\eqref{eq:smallmix5}, a nonzero right-hand side forces $x,y,z$
to equal $a$. In each case the left-hand side then equals $a^2$.
A zero right-hand side makes the inequality automatic.

We show that the five inequalities are complete. By
\eqref{eq:smallmix2}, every triple product is the same greatest
element. By \eqref{eq:smallmix1}, this element is multiplicatively
absorbing. A term containing a word of length at least three therefore
has this greatest-element normal form.

Every remaining term is a sum of singleton variables and quadratic
words. Regard $xy$ with $x\neq y$ as an edge and $x^2$ as a
loop. Delete singleton variables incident with an edge or loop by
\eqref{eq:smallmix1}. Let $C$ be the content and let $A$ be the
set of vertices incident with the quadratic words.

The odd-walk completion from Proposition~\ref{prop:smalllinear},
now using \eqref{eq:smallmix4}, applies to this graph with possible
loops. If it has no odd cycle or loop, the result is a disjoint union
of complete bipartite graphs and isolated vertices. If some component
contains an odd cycle or loop, the completion produces a square term
$x^2$. In the presence of this square, \eqref{eq:smallmix5} adds
squares at both endpoints of every quadratic word, including words in
other components. Thus every variable in $A$ acquires a square.
By \eqref{eq:smallmix3}, all the edge words are then absorbed.
The resulting form is
\begin{equation}\label{eq:smallmixedloopform}
 \sum_{x\in A}x^2+\sum_{z\in C\setminus A}z,
 \qquad A\neq\varnothing.
\end{equation}

We separate these normal forms in the two generators. The
greatest-element form is distinguished from every graph form by assigning
$a$ to all variables in $S_c^*(a^2)$: a graph form has value $a$
or $a^2$, both nonzero. Graph forms with different contents are
separated in the same semiring by assigning zero to a variable present
only in one form and $a$ to all other variables.

Two forms \eqref{eq:smallmixedloopform} with the same content but
different sets $A$ are also separated in $S_c^*(a^2)$. Choose a
variable in one set $A$ but not the other, assign $a^2$ to that
variable, and assign $a$ to all others. The form in which it is
squared becomes zero; the other form has value $a^2$.

A loop form is identically zero in $S_c^*(ab)$, whereas any bipartite
graph form has a nonzero value under a proper two-coloring. Finally,
two distinct completed bipartite graph forms of the same content are
separated in $S_c^*(ab)$ by coloring an edge present only in one
form monochromatically while properly coloring the other, as in
Proposition~\ref{prop:smalllinear}. Thus no two distinct normal
forms have the same term function on both generators. Every identity
of the join has equal normal forms and follows from
\eqref{eq:smallmix1}--\eqref{eq:smallmix5}.
\end{proof}

\begin{theorem}\label{thm:finitejoin}
For all finite integers $n,k\geq1$, the variety
\[
 \V(\Spow{n})\vee\V(\Slin{k})
\]
is finitely based if and only if $k\leq2$ or $n\geq2k-1$.
If $n\geq2k-1$, this join equals $\V(\Spow{n})$.
\end{theorem}

\begin{proof}
The equality in the last assertion is
Theorem~\ref{thm:containment}, and the resulting variety is finitely
based by Theorem~\ref{thm:finitepowers}. If $k=1$, this covers
every $n$. If $k=2$, it covers $n\geq3$; the case $n=1$ is
$\V(S_c^*(ab))$ by the inclusion of the one-letter subsemiring,
and the case $n=2$ is Lemma~\ref{lem:smallmixed}. Hence all cases
with $k\leq2$ are finitely based.

It remains to consider $k\geq3$ and $n\leq2k-2$. Suppose that
the displayed join has a finite basis. Normalize it by
Lemma~\ref{lem:normalization} to obtain a finite family $\Delta$
of inequalities $\bq\preceq\bu$, with $\bq$ a word.
Choose $N\geq1$ so that $|\bu|\leq N$ for every member of
$\Delta$. Every such inequality holds in $\Slin{k}$.

By Theorem~\ref{thm:EH}, choose a finite $k$-uniform hypergraph
$\HH$ with no isolated vertices, with
\[
 g(\HH)>\max\{N,3\},\qquad \chi(\HH)>k.
\]
Choose a vertex $v$ and a hyperedge $e$ containing it. We claim that
\begin{equation}\label{eq:squarehypergraph}
 x_v^2\preceq\bt_{\HH}
\end{equation}
holds in both generating semirings. In $\Slin{k}$, the upper
term is identically zero by Lemma~\ref{lem:hypergraphvalue}, since
$\HH$ has no strong $k$-coloring. In $\Spow{n}$,
Lemma~\ref{lem:thresholdinequality}, applied to $e$, gives
\[
 x_v^2\preceq_{\Spow{n}}\prod_{u\in e}x_u
 \preceq_{\Spow{n}}\bt_{\HH}.
\]
This proves the claim. Thus \eqref{eq:squarehypergraph} is an
identity of the join and must follow from $\Delta$.

Every word of $\bt_{\HH}$ is linear and has content a hyperedge.
Lemma~\ref{lem:preservation} applies to every elementary deduction
from $\Delta$ and preserves this property along a deduction starting
at $\bt_{\HH}$. The term $\bt_{\HH}+x_v^2$ contains a
nonlinear word, and so cannot be its endpoint. This contradicts the
derivability of \eqref{eq:squarehypergraph}. The join is therefore
nonfinitely based.
\end{proof}

\section{Infinite powers and infinite joins}
\label{sec:infinitejoins}

We now identify the variety generated by all finite powers and determine
the mixed joins involving the unrestricted linear family.

\begin{theorem}\label{thm:infinitepower}
One has
\begin{equation}\label{eq:infinitepower}
 \V(\Spowinf)
 =\bigvee_{n\geq1}\V(\Spow{n})
 =\V((\mathbb N_0,\max,+)).
\end{equation}
This variety is nonfinitely based.
\end{theorem}

\begin{proof}
For each $n\geq1$, define
\[
 \pi_n:\Spowinf\longrightarrow\Spow{n},\qquad
 \pi_n(a^j)=
 \begin{cases}
 a^j,&j\leq n,\\
 0,&j>n,
 \end{cases}
 \qquad \pi_n(0)=0.
\]
Formula \eqref{eq:poweroperations} shows that $\pi_n$ preserves
addition and multiplication, and it is surjective. The family of maps
separates elements: two distinct powers, or a power and zero, are
separated by any sufficiently large truncation. Hence the diagonal map
into $\prod_{n\geq1}\Spow{n}$ is an injective homomorphism with
surjective coordinate maps. Each finite member is a homomorphic image
of $\Spowinf$, and the infinite member is a subsemiring of their
product. These two observations prove the first equality in
\eqref{eq:infinitepower}.

The correspondence $a^j\mapsto j$, $0\mapsto\infty$ is an
isomorphism
\[
 \Spowinf\cong(\mathbb N_{>0}\cup\{\infty\},\max,+),
\]
where $\infty$ is absorbing for both operations. We compare its
identities with those of $(\mathbb N_0,\max,+)$ in the signature
$(+,\cdot)$.

Every term function on numerical arguments is a maximum of finitely
many homogeneous linear functions with nonnegative integer coefficients.
If an identity holds on positive integers, positive integer scaling
first extends it to positive rational arguments. Continuity then
extends it to all nonnegative real arguments, and in particular to
$\mathbb N_0$. The converse is immediate by restriction. Thus
positive and nonnegative integer max-plus semirings have the same
identities in this signature.

Every such identity is regular: if a variable occurs on only one side,
fix all other variables at $1$ and let this variable increase. The
side containing it is unbounded, while the other side is fixed, a
contradiction. A regular identity remains valid on adjoining the
absorbing element $\infty$. An assignment using $\infty$ on a
variable in the common content makes both sides equal to $\infty$;
all other assignments are already covered. Conversely, validity in the
extension implies validity in its positive-integer subsemiring. This
proves the second equality in \eqref{eq:infinitepower}.

Aceto, \'Esik and Ing\'olfsd\'ottir proved that the max-plus algebra
of nonnegative integers has no finite equational basis
\cite[Theorem~4.1]{AcetoEsikIngolf}. Their signature also names the
numerical constant zero, which is an identity for both operations.
The result implies nonfinite basability in our binary signature as
follows. If a finite binary basis existed, adjoining the two
common-identity laws for that constant would give a finite basis in
the expanded signature. Indeed, those laws reduce any expanded term
to the constant alone or to a constant-free term. An identity between
the constant and a nonconstant term fails when all variables have
numerical value $1$. Identities between two constant-free terms follow
from the assumed binary basis, and the constant identity is trivial.
This contradicts the cited theorem and completes the proof.
\end{proof}

\begin{corollary}\label{cor:infinitepowerabsorbs}
For every $k\geq1$,
\begin{align*}
 \V(\Spowinf)\vee\V(\Slin{k})&=\V(\Spowinf),\\
 \V(\Spowinf)\vee\V(\Slininf)&=\V(\Spowinf).
\end{align*}
These joins are nonfinitely based.
\end{corollary}

\begin{proof}
Theorem~\ref{thm:containment} places $\Slin{k}$ in
$\V(\Spow{2k-1})$, which is contained in $\V(\Spowinf)$
by Theorem~\ref{thm:infinitepower}. Taking the join over all $k$
and applying Corollary~\ref{cor:linearinfjoin} proves the second
equality. Nonfinite basability follows from
Theorem~\ref{thm:infinitepower}.
\end{proof}

For finite powers joined with the unrestricted linear family, the
finite basis property has the opposite outcome.

\begin{lemma}\label{lem:degreereduction}
For each $n\geq1$, both $\Spow{n}$ and $\Slininf$ satisfy
\begin{equation}\label{eq:degreereduction}
 x_1\cdots x_{n+2}
 \approx
 \sum_{i=1}^{n+2}x_1\cdots\widehat{x_i}\cdots x_{n+2},
\end{equation}
where the hat denotes omission of the indicated variable. Modulo this
identity, every term is equivalent to a sum of words of length at
most $n+1$.
\end{lemma}

\begin{proof}
In $\Spow{n}$, both sides have constant value zero, because each
word has length at least $n+1$. In $\Slininf$, the two sides
have the same content and the same graph normal form: every pair of
variables occurs together in a summand on the right. Since $n+2\geq3$,
such a summand exists for every pair. Theorem~\ref{thm:unrestrictedbasis}
therefore gives the identity.

To reduce a word of length $r\geq n+2$, apply
\eqref{eq:degreereduction} to any $n+2$ of its occurrences and
retain the remaining occurrences as a multiplicative context. Each
resulting word has length $r-1$. Repeatedly applying this operation
to words of length greater than $n+1$ gives the asserted reduction.
Occurrences can be chosen even when some variables repeat, since
\eqref{eq:degreereduction} admits arbitrary substitutions.
\end{proof}

\begin{theorem}\label{thm:infinitemixed}
For every finite integer $n\geq1$, the variety
\[
 \V(\Spow{n})\vee\V(\Slininf)
\]
is finitely based.
\end{theorem}

\begin{proof}
Use the commutative ai-semiring laws and
\eqref{eq:degreereduction} as background identities. By
Lemma~\ref{lem:degreereduction} and then
Lemma~\ref{lem:normalization}, it suffices to derive each valid
inequality $\bq\preceq\bu$ in which $\bq$ is a word and all
words on either side have length at most $n+1$.
Such an inequality holds in each of $\Spow{n}$ and $\Slininf$.
We choose two additive subterms of $\bu$, one for each generator.

For $\Spow{n}$, if $\bu$ contains a word of length $n+1$,
let $\bu_1$ be that single word. It has constant value zero, so
$\bq\preceq_{\Spow{n}}\bu_1$. Otherwise all upper words have
length at most $n$. Lemma~\ref{lem:powerselection} then gives
an additive subterm $\bu_1$ with
$\bq\preceq_{\Spow{n}}\bu_1$ and with the bound
\eqref{eq:powerbound} on its variables together with those of
$\bq$. In either case the number of variables used in $\bu_1$
is bounded solely in terms of $n$.

For $\Slininf$, suppose first that $\bu$ contains a nonlinear
word. Choose it as $\bu_2$; then
$\bq\preceq_{\Slininf}\bu_2$ by
Lemma~\ref{lem:unrestrictedcriterion}. Suppose instead that every
word of $\bu$ is linear. The same lemma says that $\bq$ is
linear, every variable of $\bq$ occurs in $\bu$, and every pair
of distinct variables of $\bq$ occurs together in some word of
$\bu$. For each variable choose one such upper word containing it,
and for each pair choose an upper word containing both. Let $\bu_2$
be the sum of these chosen words. Applying the criterion again gives
\[
 \bq\preceq_{\Slininf}\bu_2.
\]
Since $\ell(\bq)\leq n+1$, at most
\[
 (n+1)+\binom{n+1}{2}
\]
words have been selected, each of length at most $n+1$.

Now set $\bu_0=\bu_1+\bu_2$, deleting repetitions. It is an
additive subterm of $\bu$. It satisfies
$\bq\preceq\bu_0$ in both generating semirings, since its value
is at least the value of either selected subterm. Hence that inequality
holds in their join. The total number of variables in $\bq$ and
$\bu_0$ is bounded, for example, by
\begin{equation}\label{eq:infinitemixedbound}
 n\binom{2n+1}{n+1}
 +(n+1)\left(n+3+\binom{n+1}{2}\right).
\end{equation}
The first term bounds the contribution from the power selection; the
second accommodates $\bq$, a selected long or nonlinear word, and
the variable and pair witnesses for the unrestricted linear semiring.

Take all inequalities valid in the join on a fixed alphabet of size
\eqref{eq:infinitemixedbound}, with all words of length at most
$n+1$. This is a finite family. After renaming variables, it contains
every selected inequality $\bq\preceq\bu_0$. Each such inequality
implies $\bq\preceq\bu$ because $\bu_0$ is an additive subterm.
Together with \eqref{eq:degreereduction} and the background laws,
this finite family consequently derives every identity of the join.
\end{proof}

\begin{corollary}\label{cor:cofinaljoins}
For every finite $n\geq1$,
\begin{equation}\label{eq:cofinaljoins}
 \bigvee_{k\geq1}
 \bigl(\V(\Spow{n})\vee\V(\Slin{k})\bigr)
 =\V(\Spow{n})\vee\V(\Slininf)
\end{equation}
is finitely based, although all members on the left with
$k\geq3$ and $n\leq2k-2$ are nonfinitely based.
\end{corollary}

\begin{proof}
The equality is the associativity of joins together with
Corollary~\ref{cor:linearinfjoin}. Its right-hand side is finitely
based by Theorem~\ref{thm:infinitemixed}. The indicated finite
members are nonfinitely based by Theorem~\ref{thm:finitejoin}.
For fixed $n$, these members form a cofinal tail of the family.
\end{proof}

\section{Intervals and a limit subvariety}
\label{sec:intervals}

We first extract interval statements from the preservation lemma. We
then place the known limit variety generated by $TR_6$ below
$\V(S_c^*(abc))$ by an explicit quotient construction.

\begin{theorem}\label{thm:linearinterval}
For integers $3\leq k\leq m<\infty$, every variety in the interval
\[
 [\V(\Slin{k}),\V(\Slin{m})]
\]
is nonfinitely based.
\end{theorem}

\begin{proof}
Let a variety in this interval have a finite basis, and normalize the
basis to a finite family $\Delta$ of inequalities $\bq\preceq\bu$.
Choose $N\geq1$ bounding $|\bu|$ over this family. Each member
of $\Delta$ holds in the lower generator $\Slin{k}$.

Choose a finite $k$-uniform hypergraph $\HH$, without isolated
vertices, such that
\[
 g(\HH)>\max\{N,3\},\qquad\chi(\HH)>m.
\]
Such a hypergraph exists by Theorem~\ref{thm:EH}. We show that
$\bt_{\HH}$ has constant value zero in $\Slin{m}$. Suppose
instead that an assignment gives it a nonzero value. By
Lemma~\ref{lem:values}, every edge product is nonzero and the
values of its vertex variables have pairwise disjoint contents.
Every vertex belongs to an edge and hence has a nonzero value.
Choose one letter from each vertex value and color the vertex by that
letter's index in $[m]$. On any hyperedge these choices are distinct,
because the corresponding contents are disjoint. This is a strong
$m$-coloring and therefore an ordinary proper $m$-coloring of
$\HH$, contradicting $\chi(\HH)>m$.

Consequently, $x_v^2\preceq\bt_{\HH}$ holds in $\Slin{m}$
for any vertex $v$, and hence in every variety in the stated interval.
It must therefore follow from $\Delta$. But the initial term
$\bt_{\HH}$ satisfies $(P_{\HH})$, and
Lemma~\ref{lem:preservation}, using the validity of $\Delta$ in
$\Slin{k}$, preserves this property throughout any deduction.
The endpoint $\bt_{\HH}+x_v^2$ contains a square and fails it.
This contradiction proves the theorem.
\end{proof}

\begin{theorem}\label{thm:mixedinterval}
Let $k\geq3$ and $1\leq n\leq2k-2$. Every variety in
\[
 [\V(\Slin{k}),\,
   \V(\Spow{n})\vee\V(\Slin{k})]
\]
is nonfinitely based.
\end{theorem}

\begin{proof}
Suppose a variety in the interval has a finite normalized basis
$\Delta$, and let $N$ bound the number of words in its upper terms.
All these identities hold in $\Slin{k}$. Choose a $k$-uniform
hypergraph of girth greater than $\max\{N,3\}$ and chromatic
number greater than $k$, deleting isolated vertices. As in the proof
of Theorem~\ref{thm:finitejoin},
$x_v^2\preceq\bt_{\HH}$ holds in both $\Spow{n}$ and
$\Slin{k}$, so it holds in the upper endpoint and thus in the
chosen variety. The local preservation lemma applied to $\Delta$
prevents its derivation from $\bt_{\HH}$, since the endpoint
contains $x_v^2$. This contradicts that $\Delta$ is a basis.
\end{proof}

\begin{corollary}\label{cor:powerintermediate}
For every $k\geq3$,
\begin{equation}\label{eq:powerintermediate}
 \begin{split}
 \V(\Spow{2k-2})
 &\subsetneq
 \V(\Spow{2k-2})\vee\V(\Slin{k})\\
 &\subsetneq\V(\Spow{2k-1}).
 \end{split}
\end{equation}
The two power-generated endpoints are finitely based, and the displayed
intermediate variety is nonfinitely based. In particular,
$\V(\Spow{2k-1})$ does not cover $\V(\Spow{2k-2})$.
\end{corollary}

\begin{proof}
The first inclusion is strict by Theorem~\ref{thm:containment},
since $\Slin{k}\notin\V(\Spow{2k-2})$.
The second inclusion follows from the same theorem and the power chain.
For strictness, consider
\[
 x^{2k-1}\approx x^{2k}.
\]
It holds in $\Spow{2k-2}$ by truncation and in $\Slin{k}$
because both sides are nonlinear words. Thus it holds in the join.
It fails in $\Spow{2k-1}$ at $x=a$, where the left-hand side
is $a^{2k-1}$ and the right-hand side is zero.
The endpoint finite basis statements are
Theorem~\ref{thm:finitepowers}; the intermediate nonfinite basis
statement is Theorem~\ref{thm:finitejoin}.
\end{proof}

The linear-word varieties themselves form a strictly ascending chain:
\begin{equation}\label{eq:linearchain}
 \V(S_c^*(a))\subsetneq\V(S_c^*(ab))
 \subsetneq\V(S_c^*(abc))\subsetneq\cdots.
\end{equation}
The natural subsemiring inclusions give containment. The inequality
$y\preceq x_1\cdots x_{k+1}$ holds in $\Slin{k}$ but fails
in $\Slin{k+1}$ by assigning $a_i$ to $x_i$ and zero to $y$.
Thus the inclusions are strict. Theorem~\ref{thm:linearinterval}
shows that every finite interval of this chain above
$\V(S_c^*(abc))$ consists entirely of nonfinitely based varieties,
whereas the join of the whole chain is finitely based by
Theorem~\ref{thm:unrestrictedbasis}.

\subsection{The six-element quotient}

Shao, Ren and Gao proved that the six-element semiring $TR_6$
is nonfinitely based and that every proper subvariety of $\V(TR_6)$
is finitely based \cite[Theorem~4.6]{ShaoRenGao}. Its operations,
with the numerical labels used there, are shown in
Table~\ref{tab:TR6}. The label $1$ denotes its common absorbing
element.

\begin{table}[htbp]
\caption{Operations in the six-element semiring}
\label{tab:TR6}
\centering
\renewcommand{\arraystretch}{1.15}
\[
\begin{array}{c|cccccc}
 +&1&2&3&4&5&6\\\hline
 1&1&1&1&1&1&1\\
 2&1&2&3&3&2&3\\
 3&1&3&3&3&3&3\\
 4&1&3&3&4&3&4\\
 5&1&2&3&3&5&3\\
 6&1&3&3&4&3&6
\end{array}
\qquad
\begin{array}{c|cccccc}
 \cdot&1&2&3&4&5&6\\\hline
 1&1&1&1&1&1&1\\
 2&1&1&1&1&1&3\\
 3&1&1&1&1&1&1\\
 4&1&1&1&1&3&1\\
 5&1&1&1&3&1&3\\
 6&1&3&1&1&3&1
\end{array}
\]
\end{table}

\begin{proposition}\label{prop:TR6quotient}
The semiring $TR_6$ is a homomorphic image of a subsemiring of
$S_c^*(abc)$. Moreover,
\begin{equation}\label{eq:TR6strict}
 \V(TR_6)\subsetneq\V(S_c^*(abc)).
\end{equation}
\end{proposition}

\begin{proof}
Consider
\[
 U=\{0,b,c,ab,ac,bc,abc\}\subseteq S_c^*(abc).
\]
This set is closed under addition: the union of the contents of any
two nonzero listed words cannot be the omitted singleton $a$.
It is also closed under multiplication. The only nonzero products of
two listed elements, up to commutativity, are
\[
 b\cdot c=bc,\qquad ab\cdot c=abc,\qquad ac\cdot b=abc.
\]
Thus $U$ is a subsemiring.

Let $\rho$ have the single nonsingleton class $\{bc,abc\}$.
For every $u\in U$, both $bc\cdot u$ and $abc\cdot u$ are
zero. For addition, if $u=0$, then both sums are zero. If $u\neq0$,
then $abc+u=abc$ and $bc+u$ is either $bc$ or $abc$.
Consequently, the two sums belong to the same $\rho$-class.
These verifications show that $\rho$ is a congruence.

The six classes are
\[
 \{0\},\ \{ab\},\ \{bc,abc\},\ \{ac\},\ \{b\},\ \{c\}.
\]
Send them, in this order, to $1,2,3,4,5,6$. The nonzero products
listed above give exactly the three unordered pairs
$\{2,6\},\{4,5\},\{5,6\}$ with product $3$ in
Table~\ref{tab:TR6}; every other product is $1$.
For addition, $\{ab\}$ lies above $\{b\}$,
$\{ac\}$ lies above $\{c\}$, and the class $\{bc,abc\}$
is above every nonzero class. Sums between the two chains give this
class, while $\{0\}$ absorbs every class. This is exactly the
addition table. Hence the map is an isomorphism $U/\rho\cong TR_6$.

It follows that $\V(TR_6)\subseteq\V(S_c^*(abc))$.
Every triple product in $TR_6$ is $1$, and every square is $1$,
so it satisfies
\[
 xyz\approx t^2.
\]
This identity fails in $S_c^*(abc)$ at $x=a,y=b,z=c,t=a$:
its values are $abc$ and $0$. Therefore the inclusion is strict.
\end{proof}

\begin{corollary}\label{cor:notminimal}
The variety $\V(S_c^*(abc))$ is not a limit variety. In particular,
it is not a minimal nonfinitely based subvariety of $\V(S_{53})$.
\end{corollary}

\begin{proof}
Proposition~\ref{prop:TR6quotient} gives a proper subvariety
$\V(TR_6)$, which is nonfinitely based by
\cite[Theorem~1.2]{ShaoRenGao}. Thus not every proper subvariety
of $\V(S_c^*(abc))$ is finitely based. Since
$S_c^*(abc)\in\V(S_{53})$ by \cite[Lemma~4.4]{Gao}, the second
assertion follows as well.
\end{proof}

\clearpage
\appendix
\section{Cayley tables of the eight-element semiring}

For $k=3$, the set $W^{\leq}\cup\{0\}$ consists of the eight
elements
\[
 0,\ a,\ b,\ c,\ ab,\ ac,\ bc,\ abc.
\]
The two tables below use this order. They are the specialization of
\eqref{eq:multiplication} and \eqref{eq:addition}. In the addition
table, the sum of two nonzero words contains all their letters, each
once. In the multiplication table, a repeated letter makes the product
zero. For example,
\[
 a+ab=ab,\qquad a+bc=abc,\qquad a\cdot ab=0,
 \qquad a\cdot bc=abc.
\]
In particular, the addition is not flat, since $a+b=ab\neq 0$.

\begin{table}[!ht]
\caption{Addition in the eight-element semiring}
\centering
\renewcommand{\arraystretch}{1.18}
\[
\begin{array}{c|cccccccc}
 +&0&a&b&c&ab&ac&bc&abc\\\hline
 0&0&0&0&0&0&0&0&0\\
 a&0&a&ab&ac&ab&ac&abc&abc\\
 b&0&ab&b&bc&ab&abc&bc&abc\\
 c&0&ac&bc&c&abc&ac&bc&abc\\
 ab&0&ab&ab&abc&ab&abc&abc&abc\\
 ac&0&ac&abc&ac&abc&ac&abc&abc\\
 bc&0&abc&bc&bc&abc&abc&bc&abc\\
 abc&0&abc&abc&abc&abc&abc&abc&abc
\end{array}
\]
\end{table}

\begin{table}[!ht]
\caption{Multiplication in the eight-element semiring}
\centering
\renewcommand{\arraystretch}{1.18}
\[
\begin{array}{c|cccccccc}
 \cdot&0&a&b&c&ab&ac&bc&abc\\\hline
 0&0&0&0&0&0&0&0&0\\
 a&0&0&ab&ac&0&0&abc&0\\
 b&0&ab&0&bc&0&abc&0&0\\
 c&0&ac&bc&0&abc&0&0&0\\
 ab&0&0&0&abc&0&0&0&0\\
 ac&0&0&abc&0&0&0&0&0\\
 bc&0&abc&0&0&0&0&0&0\\
 abc&0&0&0&0&0&0&0&0
\end{array}
\]
\end{table}

\end{document}